\documentclass[10pt, oneside]{article}

\usepackage{fancyhdr}
\fancyheadoffset{ 0.5 cm}
\fancyfootoffset{ 0.5 cm}

\usepackage{titlesec}
\titlelabel{\thetitle.\quad}

\usepackage{times}
\usepackage{geometry}
\usepackage{graphicx}
\usepackage{amssymb, amscd}
\usepackage{amsmath}
\makeatletter
\@addtoreset{equation}{section}
\makeatother
\numberwithin{equation}{subsection}
\usepackage{amsthm} 
\usepackage{mathrsfs}
\usepackage{epstopdf}
\usepackage{enumerate}
\usepackage{url}
\usepackage{graphicx}
\usepackage{mathtools}
\usepackage{tikz}							
\usepackage{setspace}
\usepackage{amsfonts, amsmath, wasysym}

\usepackage{hyperref}
\usepackage{cite}

\usepackage[draft]{optional}

\usepackage{color}

\usepackage{xcolor}
\newtagform{red}{\color{red}(}{)}

\usepackage{mathtools}

\usepackage{amsmath}
\makeatletter
\@addtoreset{equation}{section}
\makeatother
\numberwithin{equation}{section}

\theoremstyle{plain}
\newtheorem{theorem}{Theorem}[section]
\newtheorem{lemma}[theorem]{Lemma}

\newtheorem{conjecture}{Conjecture}

\theoremstyle{definition}

\theoremstyle{remark}

\newtheorem{remark}[theorem]{\bf{Remark}}

\newcommand{\cU}{\ensuremath{{\cal U}}}

\newcommand{\m}{\ensuremath{{\cal M}}}

\newcommand{\ca}{\ensuremath{{\cal A}}}

\newcommand{\cR}{\ensuremath{{\cal R}}}
\newcommand{\cH}{\ensuremath{{\cal H}}}

\newcommand{\al}{\alpha}

\newcommand{\de}{\delta}

\newcommand{\Tau}{{\cal T}}

\newcommand{\vph}{\varphi}
\newcommand{\ep}{\varepsilon}

\newcommand{\R}{\ensuremath{{\mathbb R}}}
\newcommand{\N}{\ensuremath{{\mathbb N}}}
\newcommand{\B}{\ensuremath{{\mathbb B}}}
\newcommand{\D}{\ensuremath{{\mathbb D}}}
\newcommand{\Z}{\ensuremath{{\mathbb Z}}}
\newcommand{\C}{\ensuremath{{\mathbb C}}}

\newcommand{\HH}{\ensuremath{{\mathbb H}}}

\newcommand{\ovB}{\ensuremath{{ \overline{ \mathbb B }}}}

\newcommand{\bA}{{\bf A }}
\newcommand{\bB}{{\bf B}}

\DeclareMathOperator{\VolBB}{Vol\mathbb{B}}

\DeclareMathOperator{\sech}{sech}

\newcommand{\beq}{\begin{equation}}
\newcommand{\eeq}{\end{equation}}
\newcommand{\beqa}{\begin{equation}\begin{aligned}}
\newcommand{\eeqa}{\end{aligned}\end{equation}}
\newcommand{\brmk}{\begin{rmk}}
\newcommand{\ermk}{\end{rmk}}
\newcommand{\partref}[1]{\hbox{(\csname @roman\endcsname{\ref{#1}})}}

\newcommand{\Rm}{{\mathrm{Rm}}}
\newcommand{\Ric}{{\mathrm{Ric}}}
\newcommand{\K}{{\mathrm{K}}}

\usepackage{soul}

\newenvironment{claim}[1]{\par\noindent\underline{Claim:}\space#1}{}
\newenvironment{claimproof}[1]{\par\noindent\underline{Proof of Claim:}\space#1}{\leavevmode\unskip\penalty9999 \hbox{}\nobreak\hfill\quad\hbox{$\dagger \dagger$}}

\title{Improved Pseudolocality on Large Hyperbolic Balls}
\author{Andrew D. McLeod}
\date{ \today}

\begin{document}
\usetagform{red}
\maketitle
\begin{abstract}
We obtain an improved pseudolocality result for Ricci flows on two-dimensional surfaces that are initially almost-hyperbolic
on large hyperbolic balls. We prove that, at the central point of the hyperbolic ball, the Gauss curvature remains
close to the hyperbolic value for a time that grows exponentially in the radius of the ball. 
This two-dimensional result allows us to precisely conjecture how the phenomenon should appear in the higher dimensional setting.
\end{abstract}

{\small \tableofcontents}

\section{Introduction}
A Ricci flow solution $g(t)$ on a smooth $n$-dimensional manifold $\m,$ defined for all $t \in [0,T],$ is a one-parameter family of smooth Riemannian metrics $g(t),$ for $t \in [0,T],$ on $\m$ whose evolution is governed by
the equation
\beq
	\label{rf-eqn}
		\frac{ \partial g}{\partial t}(t) = -2\Ric_{g(t)}
\eeq
with $g(0) := g_0$ for some given initial metric $g_0$ on $\m.$ The Ricci flow equation in \eqref{rf-eqn} can be viewed as a non-linear heat equation.

The powerful pseudolocality theorem of Perelman, Theorem 10.1 in \cite{Perelman}, exhibits a property of complete Ricci flow solutions of bounded curvature
which is false for solutions of the linear heat equation. Roughly speaking this theorem asserts that if a local region looks \emph{almost Euclidean} then it cannot suddenly look highly non-trivial.
There are numerous conditions that can be used to provide a precise meaning of \emph{almost Euclidean}; see, for example, the conditions utilised in any of
Theorems 10.1 and 10.3 in \cite{Perelman} and Proposition 3.1 in \cite{Tian} (though it is worth remarking that recent work of Fabio Cavalletti and Andrea Mondino 
establishes that the conditions assumed in Proposition 3.1 in \cite{Tian} imply that the hypotheses of Theorem 10.1 in \cite{Perelman} are satisfied on a strictly smaller 
initial region, see \cite{Mondino}). 

More recently, Miles Simon and Peter Topping obtain a pseudolocality-type result in dimension three valid outside the \emph{almost Euclidean} setting.
In particular, a consequence of Theorem 1.1 in \cite{Simon1} is that even when the hypotheses of Proposition 3.1 in \cite{Tian} are not close to their Euclidean counterparts,
one may still conclude $C/t$ curvature decay for some $C > 0.$

A particularly interesting consequence of pseudolocality is that, under complete flows with bounded curvature, initial curvature bounds propagate forward
for some definite period of time. This phenomenon is precisely captured by Theorem 10.3 in \cite{Perelman},
whilst the following result of Chen in \cite{Chen2} provides a similar example of the same phenomenon under weaker assumptions in dimension $2$. 

\begin{theorem}[Variant of Proposition 3.9 in \cite{Chen2}]
\label{Chenpseudo}
Let $g(t)$ be a smooth Ricci flow on a smooth surface $\m^2$ defined for all $t \in [0,T].$ 
Let $x_0 \in \m$ and assume, for some $r_0 > 0$, that for all $t \in [0,T]$ we have $\B_{g(t)} ( x_0 , r_0 ) \subset \subset \m.$
For a given $v_0 > 0$ suppose that $\left| \K_{g(0)} \right| \leq r_0^{-2} $ throughout $\B_{g(0)} (x_0 , r_0),$ and $\VolBB_{g(0)}  ( x_0 , r_0 ) \geq v_0 r_0^2.$
Then there exists a constant $A = A(v_0) > 0$ such that
$$ \forall (x,t) \in \B_{g(t)} ( x_0 , r_0/2 ) \times \left[ 0 , \min \left\{ T , Ar_0^2 \right\} \right] \qquad \text{we have} \qquad \left| \K_{g(t)} (x) \right| \leq 2r_0^{-2}. $$
\end{theorem}
\vskip 4pt
\noindent
An instructive simple setting for pseudolocality is when the initial metric is locally Euclidean on some ball.
In particular, suppose we have a complete, smooth Ricci flow $g(t)$ on a smooth surface $\m^2,$ defined for all $t \in [0,T]$ for some $T > 0,$
with $\B_{g(0)} ( x_0 , R)$ isometric to a Euclidean disc of radius $R.$
Then Theorem \ref{Chenpseudo} gives a universal $A > 0$ such that for $0 \leq t \leq \min \{ T , A R^2  \}$ we have $ \left| \K_{g(t)}(x_0) \right| \leq 2R^{-2}.$ 
Therefore the Gauss curvature $\K_{g(t)}$ at the point $x_0$ remains close to $0$ (the Euclidean Gauss curvature) for 
a time proportional to the square of the radius $R.$

In the hyperbolic setting,
namely, when we have that $ \B_{g(0)} ( x_0 , R) $ is isometric to 
a hyperbolic disc of radius $R,$ Theorem \ref{Chenpseudo} can again be applied.
However, the requirement that $| \K_{g(0)} | \leq r_0^{-2}$ throughout $\B_{g(0)} (x_0 , r_0)$ limits us to considering only radii $r_0 \in (0, 1].$
Therefore the Gauss curvature at $x_0$ may only be controlled for some fixed order one time, irrespective of how large $R$ is.

Our first main result establishes that, provided a sufficiently large initial ball is isometric to a hyperbolic disc of the same radius, the flow remains almost hyperbolic (in a scaled sense) at the central point for a time that is exponential in the radius.

\begin{theorem}[{\bf{Improved control time with equality on large initial ball}}]
\label{hypat0}
For any $ \al \in (0,1]$ there exist constants $\cR = \cR (\al ) > 0$ and $c = c(\al ) >0$ for which the following holds: 

Let $R \geq \cR$ and assume that $g(t)$ is a complete smooth Ricci flow on a smooth surface $\m,$ defined for all $t \in [0,T]$ for some $T > 0,$
and such that, for some $x \in \m,$ we have that $\left( \B_{g(0)} ( x , R) , g(0) \right)$ is isometric to a hyperbolic disc of radius $R.$ 
Then at the point $x$ we have
\beq
	\label{conc_time_0_equiv_thm}
		 -1 - \al \leq \K_{\frac{g(t)}{1+2t}}(x) \leq -1 + \al \qquad \text{for all} \qquad 0 \leq t \leq \Tau_{max} := \min \left\{ T , e^{cR} \right\}.
\eeq
\end{theorem}

\begin{remark}
\label{scaling_role}
Consider a smooth surface $\Sigma$ and a complete metric $g_{\HH}$ 
of constant Gauss curvature $-1$ on $\Sigma$.
Then there is a unique complete Ricci flow 
$\mathfrak{h}(t) := (1+2t)g_{\HH}$,
defined on $\Sigma$ for all times $t \in [0,\infty)$, 
with $\mathfrak{h}(0) \equiv g_{\HH}$.
The uniqueness, a consequence of Theorem 1.1 in \cite{Topping2}, allows us to refer to this flow as the hyperbolic Ricci flow on $\Sigma.$
The rescaled flow $\frac{\mathfrak{h}(t)}{1+2t}$
is identically equal to $h$ for all times $t \in [0,\infty)$, 
and so its Gauss curvature is identically $-1$ for all 
$t \in [0,\infty)$.
This is true irrespective of the smooth surface $\Sigma$,
and so given an arbitrary Ricci flow $G(t)$,
we can measure how hyperbolic the flow is by comparing the Gauss curvature of the rescaled flow $\frac{G(t)}{1+2t}$ with $-1$.

It is in this scaled sense that the Gauss curvature bounds in \eqref{conc_time_0_equiv_thm} establish that the flow $g(t)$ 
remains \emph{``$\al$ almost-hyperbolic"}
at $x$ for all $0 \leq t \leq \Tau_{max}$. 
That is, the Gauss curvature of the \emph{rescaled} flow 
$\frac{g(t)}{1+2t}$ remains within $\al$ of $-1$, which is the value
of the Gauss curvature of the \emph{rescaled} hyperbolic
Ricci flow $\frac{\mathfrak{h}(t)}{1+2t}$, 
for all times $0 \leq t \leq \Tau_{max}$.
In terms of the \emph{unscaled} flow $g(t)$, we have
$\left|\K_{g(t)}(x) - \K_{\mathfrak{h}(t)} \right|
= \left|\K_{g(t)}(x) + \frac{1}{1+2t} \right| 
 \leq \frac{\al}{1+2t}$
for all times $0 \leq t \leq \Tau_{max}$.
Hence not only does $\K_{g(t)}(x)$ remain within $\al$ of 
$-\frac{1}{1+2t}$, the value of the hyperbolic Ricci flows Gauss
curvature $\K_{\mathfrak{h}(t)}$, it in fact becomes closer
to this value as $t$ increases.
\end{remark}

\begin{remark}
\label{2d_existence}
Since the hyperbolic volume of a hyperbolic disc is exponential in the radius, by appealing to the well-developed two-dimensional existence theory (see Theorem 1.3 in \cite{Giesen3}), we may deduce that 
$\left( \B_{g(0)} ( x , R) , g(0) \right)$ being isometric to a hyperbolic disc of radius $R$ implies that the time $T$ for which the flow exists may be taken to be
exponential in the radius $R.$ Therefore $\Tau_{max}$ in \eqref{conc_time_0_equiv_thm} can be taken to be exponential in the radius $R.$
\end{remark}

\begin{remark}
\label{not_complete}
The completeness hypothesis can be weakened. 
The precise condition may be found in Theorem \ref{over_arch_thm}.
Roughly, it requires $g(t)$ balls centred at points $z \in \B_{g(0)} (x , R)$ to remain compactly contained within $\m,$ with the radius of the ball depending on the
$g(0)$ distance of $z$ from $\partial \B_{g(0)} (x, R).$ Of course a complete flow will automatically satisfy this condition. 
Finally, we do not require the flow $g(t)$ to be of bounded curvature,
which will later be seen as a direct consequence of Theorem \ref{Chenpseudo} being valid for flows of unbounded curvature.
\end{remark}
\vskip 2pt
\noindent
Since the pseudolocality result of Chen, Theorem \ref{Chenpseudo}, 
is applicable when the Gauss curvature of the initial metric 
$g(0)$ is only close to the Gauss curvature of the
hyperbolic metric it is natural to wonder if our result remains 
valid under weakened almost-hyperbolic initial assumptions. 
The global situation suggests this should be the case.
It is known that for Ricci flows conformally equivalent to complete hyperbolic metrics, if the initial metric is, in some sense, globally hyperbolic-like then the flow remains $C^l$ close to the hyperbolic Ricci flow
over its entire existence time. For example, see Theorem 2.3 in \cite{Giesen3}, and the subsequent discussion illustrating that the flows considered within this result may be 
extended to exist for all times $t \in [0,\infty).$

Naturally, without assuming the desired Gauss curvature closeness at time $t=0,$ there must be some time delay before such an estimate becomes valid.
Therefore we are led to expecting the result of Theorem \ref{hypat0} to be true, after an arbitrary short time delay, under weaker almost-hyperbolic 
assumptions at time $t=0.$ Our second main result verifies this expectation.

\begin{theorem}[{\bf{Improved control time under almost-hyperbolic hypotheses}}]
\label{mainresult}
There is a universal $\ep >0$ such that for any $ \al \in (0,1]$ and any $\de \in (0,\ep)$ there exist
constants $b = b(\al , \de) \in (0,1),$ $c = c(\al , \de) > 0$ and $\cR = \cR (\al , \de) > 0$ for which the following holds: 

Assume $R \geq \cR$ and that $(\m , \cH)$ is a smooth surface with $\B_\cH (x,R) \subset \subset \m$ for some $x \in \m$ and 
$\left( \B_\cH ( x , R) , \cH \right)$ is isometric to a hyperbolic disc of radius $R.$
Suppose $g(t)$ is a complete smooth Ricci flow on $\m,$ defined for all $t \in [0,T]$ for some $T > 0,$ with $g(0)$ conformal to $\cH$
and satisfying that
\beq
	\label{main_thm_assumed_ests}
		(\bA) \quad (1-b)\cH \leq g(0) \leq (1+b)\cH \qquad \text{and} \qquad (\bB) \quad | \K_{g(0)} | \leq 2
\eeq
throughout $\B_\cH (x,R).$
Then at the point $x \in \m$ we have
\beq
	\label{conc_of_main_thm_2}  
		-1 - \al \leq \K_{\frac{g(t)}{1+2t}} (x) \leq -1 + \al 
		\qquad \text{for all} \qquad 
		\de \leq t \leq \Tau_{max} := 
		\min \left\{ T ,  e^{cR}  \right\}.
\eeq
\end{theorem}

\begin{remark}
\label{no_interp}
It may initially appear that, by choosing $b$ sufficiently small,
we could use a combination of PDE regularity theory and interpolation
to combine $(\bA)$ and $(\bB)$ in \eqref{main_thm_assumed_ests}
to yield $C^2$-closeness between the metrics $g(0)$ and $\cH$
throughout a smaller ball $\B_{\cH}(x,R - \gamma)$ for some $\gamma \in (0,1)$, say.
Using the good $L^{\infty}$-estimates
provided by $(\bA)$ to obtain $C^2$-estimate for $g(0) - \cH$
would require, in particular, $C^{2,\al}$-estimates for $g(0)$.
To avoid introducing dependence on the particular metric $g(0)$
we would need to establish these $C^{2,\al}$-estimates
from the hypotheses in \eqref{main_thm_assumed_ests}, rather 
than directly appealing to the smoothness of $g(0)$.

In terms of a conformal factor $u$ for $g(0)$ 
(cf. Section \ref{2d_advantage}), 
the Gauss curvature bound $(\bB)$ in \eqref{main_thm_assumed_ests}
provides estimates for $\Delta u$, 
and one might expect that combining these with elliptic
regularity theory, and the $L^{\infty}$-estimates provided
by $(\bA)$, will yield the desired $C^{2,\al}$-estimates.
However, $(\bB)$ only gives the
pointwise inequality that
$|\Delta u | \leq 2e^{2u}$ throughout $\B_{\cH} (x,R)$.

Consequently, H\"{o}lder bounds and derivative bounds for 
$u$ are not inherited by
$\Delta u$ (as they would be if we knew $\Delta u = F(u)$ for 
a suitably well-behaved function $F:\R \to \R$, say), 
and thus we cannot bootstrap standard elliptic regularity theory
or appeal to Schauder theory (see \cite{GT98}, for example) to 
obtain the $C^{2,\al}$-estimates required for interpolation.
Therefore we cannot conclude any uniform $C^2$-closeness of $g(0)$ 
to $\cH$, and instead must expect there to be an arbitrarily short
time delay before the parabolic smoothing affect of the flow can
provide such uniform $C^2$ control.
\end{remark}

\begin{remark}
If $T < \de$ then \eqref{conc_of_main_thm_2} is vacuous. However, the first estimate in \eqref{main_thm_assumed_ests} coupled with the fact 
that the hyperbolic volume
of a hyperbolic disc is exponential in the radius yield that, for sufficiently large $R,$ we have that $\VolBB_{g(0)} (x,R) \geq e^{aR}$ 
for some universal $a > 0.$
Therefore, as in Remark \ref{2d_existence}, the time $\Tau_{max}$ in \eqref{conc_of_main_thm_2} can be taken to be exponential in the radius $R.$
\end{remark}

\begin{remark}
The Gauss curvature bound \eqref{conc_of_main_thm_2} yields that, 
after an arbitrarily small delay,
the flow $g(t)$ becomes \emph{``$\al$ almost-hyperbolic"} 
at $x$, in the scaled sense detailed in Remark \ref{scaling_role},
and remains so until time $\Tau_{max}$. 
\end{remark}

\begin{remark}
The time $t=0$ Gauss curvature bound of $|\K_{g(0)}| \leq 2$ throughout $\B_\cH (0,R)$ could be weakened to being bounded by some $K_0 > 0.$ 
However, the constant $\ep > 0$ would now depend on $K_0,$ and we necessarily have to allow all the constants $b,$ $c$ and $\cR$ to additionally depend on $K_0.$ 
\end{remark}

\begin{remark}
We do not require the flow $g(t)$ to have bounded curvature, and the completeness hypothesis may once again be weakened as alluded to in 
Remark \ref{not_complete}.
\end{remark}
\vskip 2pt
\noindent
Theorem \ref{mainresult} will be obtained via an iterative procedure 
that we now outline. Lemma \ref{Barriers} is both the first step and 
the main novelty.
It establishes that the rescaled flow $\frac{g(t)}{1+2t}$
satisfies the same barrier bounds as assumed in $(\bA)$ of 
\eqref{main_thm_assumed_ests} on a
smaller closed ball $\overline{\B}_{\cH}(x,R-J)$ 
for a definite amount of time $\ep > 0.$
It will be crucial that we are allowed to make $J$ large; 
moving a sufficiently far distance away from the original boundary 
is essential to establishing the persistance of the barriers 
for the rescaled flow.

The second step makes use of the fact that in two-dimensions the Ricci flow equation reduces to a quasilinear PDE for the conformal factor.
The new barriers for the rescaled flow
$\frac{g(t)}{1+2t}$ provide improved $L^{\infty}$-bounds for the conformal factor over $\overline{\B}_{\cH}(x,R-J) \times [0,\ep]$.
Standard quasilinear PDE theory then gives, away from the parabolic boundary, derivative bounds for the conformal factor. 
Interpolating between the good $L^{\infty}$-bounds and the third order derivative bounds, for example, allows us to conclude the desired
Gauss curvature control on $\overline{\B}_{\cH}(x,R-J -2) \times 
[\de , \ep]$ for arbitrary $\de \in (0,\ep)$.
These estimates could be obtained throughout 
$\overline{\B}_{\cH}(x,R - J - \gamma)$
for any $0 < \gamma < R -J$ at the cost of introducing dependence
on $\gamma$; our choice of taking $\gamma = 2$ is only for convenience.

After replacing the original $J$ by $J+2$, we can combine these two steps to establish that the rescaled flow $\frac{g(t)}{1+2t}$ 
satisfies the 
barriers assumed in $(\bA)$ of \eqref{main_thm_assumed_ests} on 
$\overline{\B}_{\cH}(x,R-J) \times [0,\ep]$, whilst the Gauss curvature 
of $\frac{g(t)}{1+2t}$ is bounded between $-1-\al$ and $-1+\al$ on
$\overline{\B}_{\cH}(x,R-J) \times [\de , \ep]$. In particular, this means that the metric $\frac{g(\ep)}{1+2\ep}$ satisfies both 
$(\bA)$ and $(\bB)$ of \eqref{main_thm_assumed_ests}.
Thus if consider times $t \geq \ep$ and rescale the original flow 
$g(t)$ to a flow that takes 
$\frac{g(\ep)}{1+2\ep}$ as its initial metric, we can repeat the 
above procedure for the rescaled flow.

We proceed by iterating this procedure as many times as possible, until either the flows existence time is reached or 
the radius of the resulting ball becomes too small to support a 
further application. The explicit form of $\Tau_{max}$ is obtained by 
carefully tracking the number of iterations that can be made and for 
how long the estimates are established on each iteration. 
The rescaling involved ensures that each subsequent iteration 
establishes the barriers for $\frac{g(t)}{1+2t}$ over a longer time interval than the previous step did. This successive increasing 
leads to $\Tau_{max}$ growing exponentially with respect to the radius. 

After iterating as many times as possible, the rescaled flow
$\frac{g(t)}{1+2t}$ satisfies the barrier estimates of $(\bA)$ 
in \eqref{main_thm_assumed_ests} up until time $\Tau_{max}.$
An additional combination of the barrier estimates and standard quasilinear PDE theory gives the Gauss curvature estimates in \eqref{conc_of_main_thm_2}. 
The time delay in Theorem \ref{mainresult} is a consequence of 
the quasilinear PDE theory only being valid away from the parabolic boundary. In Theorem \ref{hypat0}, the equality at $t=0$ gives 
us estimates at $t=0$ that allow us to invoke
variants that do not require 
moving away from the time $t=0$ portion of the parabolic boundary,
enabling us to avoid any time-delay before controlling the 
Gauss curvature.
This single additional step is the only difference between the 
proofs of Theorems \ref{hypat0} and \ref{mainresult}.

The techniques used to prove our main results exploit many advantageous facts about Ricci flow specific to dimension $2$ (cf. Section \ref{2d_advantage}).
Hence they cannot generalise to higher dimensions. However, there are no obvious non-artificial obstructions to the higher dimensional analogues, and we make the following
conjecture that the same phenomenon is valid in higher dimensions.

\begin{conjecture}[{\bf{Improved time control with equality on initial ball}}]
Let $n \in \N$ such that $n \geq 3.$ There are constants $\ca = \ca(n) > 0,$ $c = c(n) > 0$ and $\cR = \cR (n) >0$ 
for which the following holds: 

Let $R \geq \cR$ and suppose that $g(t)$ is a smooth complete Ricci flow of bounded curvature on a smooth $n$-dimensional manifold $\m,$
defined for all $t \in [0,T]$ for some $T > 0,$ and, for some $ x \in \m,$ suppose we have that $\left( \B_{g(0)} (x,R) , g(0) \right)$ is isometric to
a hyperbolic ball of radius $R.$
Then at $x\in \m$ we have that 
$$ | \Rm |_{g(t)}(x) \leq \ca \qquad \text{for all} \qquad 0 \leq t \leq \Tau_{max} := \min \{ T , e^{c R} \}.$$ 
\end{conjecture}
\vskip 2pt
\noindent 
We further expect that the hypotheses of the previous conjecture can be weakened to almost-hyperbolic hypotheses in a similar spirit to the hypotheses of Theorem \ref{mainresult}.
The remainder of the paper is structured as follows. In Section \ref{2d_advantage} we collect together several well-known facts about two-dimensional Ricci flow and hyperbolic geometry.
In Section \ref{preserve_lemmata} we prove several supplementary lemmata recording how (and in what sense) our local almost-hyperbolic hypotheses are preserved under Ricci flow.
Finally in Section \ref{main_results_obtained} we provide proof of both Theorem \ref{hypat0} and Theorem \ref{mainresult}. 
In fact, both are consequences of Theorem \ref{over_arch_thm}. 

\vskip 4pt
\noindent
\emph{Acknowledgements:} 
This work was supported by EPSRC doctoral fellowship EP/M506679/1. The author would like to thank Peter Topping for numerous discussions on this topic.

\section{Preliminary Material}
\label{2d_advantage}
On a smooth two-dimensional surface we have that $\Ric_g = \K_g \cdot g.$ 
Thus the Ricci flow equation \eqref{rf-eqn} becomes 
\beq
	\label{'2dric} \frac{\partial}{\partial t} g(t) = -2 \K_{g(t)} \cdot g(t).
\eeq
Therefore the Ricci flow moves within a fixed conformal class.
If we pick a local isothermal complex coordinate $z = x+ iy$ on $U \subset \m$ we can write the metric (on $U$) as 
$g = e^{2u} |dz|^2$ for a scalar conformal factor $u \in C^{\infty} (U).$
A computation shows that, under Ricci flow, the metric's conformal factor satisfies 
\beq
	\label{2dric} 
		\frac{\partial u}{\partial t}  = e^{-2u} \Delta u = - \K_{g(t)} 
\eeq 
where $\Delta := \frac{\partial^2}{\partial x^2} + \frac{\partial^2}{\partial y^2}$ is defined with respect to the local coordinate $z=x+iy.$

Let $h$ be the complete conformal metric of constant Gauss curvature $-1$ on $\D := \{ z \in \C : |z| < 1 \}$ 
which may be globally written as $h = e^{2\vph} |dz|^2$ where $\vph (z) := \log \frac{2}{1-|z|^2}.$ 
Throughout we work on smooth surfaces $(\m , \cH)$ that contain a point $x \in \m$ such that for some $R > 0$ the ball $\B_\cH (x,R) \subset \subset \m$ and we have that
$\left( \B_\cH (x,R) , \cH \right)$ is isometric to a hyperbolic disc of radius $R,$ i.e. to $( \B_h (0,R) , h ).$
Clearly any smooth Ricci flow $g(t)$ defined on $\B_{\cH} (x,R)$ for all $t \in [0,T]$ may be viewed as a smooth Ricci flow defined on $\B_h(0,R)$ for all $t \in [0,T].$

Suppose that, for some $w \in \D$ and $r >0,$ we have a smooth Ricci flow $g(t)$ defined on $\B_h (w,r)$ for all $t \in [0,T].$ 
By choosing a local isothermal complex coordinate $z,$ we can write $g = e^{2u} |dz|^2$ throughout $\B_h (w,r) \times [0,T]$
for a smooth scalar function $u : \B_h (w,r) \times [0,T] \to \R.$
Choosing a different local isothermal complex coordinate will induce a different conformal factor,
however, the difference of two conformal factors is invariantly defined. 

Given any $w \in \D$ we may choose a 
M\"{o}bius diffeomorphism $\Psi$
(an isometry of $\D$ with respect to the hyperbolic metric $h$)
mapping $0$ to $w.$
Consider smooth metrics $g_1 = e^{2\phi_1}|dz|^2$ and $g_2 = e^{2\phi_2}|dz|^2$ on $\B_h (w,r)$ for some $r > 0,$
and assume that for $A,B > 0$ we have 
$Ag_1 \leq g_2 \leq Bg_1$ throughout $\B_h (w,r)$.
Then the pullback metrics $\Psi^{\ast}g_1$ and $\Psi^{\ast}g_2$ are defined throughout $\B_h(0,r)$ and satisfy that
$A\Psi^{\ast}g_1 \leq \Psi^{\ast}g_2 \leq B\Psi^{\ast} g_1$
throughout $\B_h(0,r)$.

In general, the new conformal factors $\hat{\phi}_1$ and 
$\hat{\phi}_2$, for the pullback metrics $\Psi^{\ast}g_1$ and 
$\Psi^{\ast}g_2$ respectively, 
will be different from $\phi_1$ and $\phi_2$.
However, we will have the same pointwise bounds for 
$\hat{\phi}_1 - \hat{\phi}_2$ throughout $\B_h(0,r)$
as we had for $\phi_1 - \phi_2$ throughout $\B_h(w,r)$.
That is, from $Ag_1 \leq g_2 \leq Bg_1$ throughout $\B_h (w,r)$ 
we may conclude that 
$\log A + 2\phi_1 \leq 2\phi_2 \leq \log B + 2\phi_1$
throughout $\B_h (w,r)$.
Since we still have that
$A\Psi^{\ast}g_1 \leq \Psi^{\ast}g_2 \leq B\Psi^{\ast} g_1$
throughout $\B_h(0,r)$ we can still conclude that
$ \log A + 2\hat{\phi}_1 \leq 2 \hat{\phi}_2 
\leq \log B + 2\hat{\phi}_1$ throughout $\B_h(0,r)$.
This will frequently be exploited, in the special case that 
$g_1 \equiv h$, to reduce working near a point 
$w \in \D$ to working near the origin $0 \in \D$.

Frequently it will be convenient to switch between the hyperbolic distance from $0$ and the Euclidean distance from $0$ on $\D.$ 
For any $z \in \D$ we have $d_h (0,z) =  \log \left[ \frac{1+|z|}{1-|z|} \right] =  2\tanh^{-1} (|z|)$ 
and hence $\B_h (0,R) = \D_{ \tanh (R/2)}.$
Here we use the notation that $\D_{\rho} := \left\{ z \in \D : |z| < \rho \right\}$ for $0 < \rho < 1.$ 
With a view to later requiring lower bounds on certain radii, we record the following elementary lower bound for $\tanh.$

\begin{lemma}[Elementary lower bound for $\tanh$]
\label{tanh_bound}
For any $x \in (0,\infty)$ we have the lower bound 
\beq
	\label{tanh_ineq} 
		\tanh (x) \geq 1 - \frac{1}{x}.
\eeq
\end{lemma}

\begin{proof}[Proof of Lemma \ref{tanh_bound}]
Define $F : (0,\infty) \to (0,1)$ by $F(x) := x \tanh (x) - x + 1.$ It suffices to establish that $F(x) \geq 0$ throughout $(0,\infty).$
Since $\tanh(x) > 0$ on $(0,\infty)$ it is apparent that $F(x) > 0$ for every $x \in (0,1).$
For $x \geq 1$ we compute the derivative of $F$ and observe
$$ F'(x) = \tanh(x) - 1 + x \sech^2 (x) = \frac{ (4x-2)e^{2x} - 2}{(e^{2x} + 1)^2} \geq 0.$$
Thus, for $x \geq 1,$ we have that $F(x) \geq F(1) = \tanh (1) >0.$ 
Therefore $F(x) > 0$ for all $x \in (0,\infty).$ 
\end{proof}
\vskip 2pt
\noindent
Finally we recall the following elementary weak comparison principle, found in \cite{Giesen}, for example. 

\begin{theorem}[Elementary comparison principle; Theorem 2.3.1 in \cite{Giesen}]
\label{comp princ}
Let $\mathcal{U} \subset \mathbb{C}$ be an open, bounded domain and, for some $T>0,$ suppose 
$w , v \in C^{\infty} ( \overline{\cU} \times [0,T])$
are both solutions to $\frac{\partial \psi}{\partial t} = e^{-2\psi} \Delta \psi$ throughout $U \times [0,T].$ If $v(z,0) \geq w(z,0)$ throughout $\cU$
and $v(z,t) \geq w(z,t)$ throughout $\partial \cU \times [0,T]$ then we may conclude that $v(z,t) \geq w(z,t)$ throughout $\overline{\cU} \times [0,T].$ 
\end{theorem}

\section{Hyperbolic Preservation Lemmata}
\label{preserve_lemmata}
Throughout, when referring to metric balls we use the convention that those denoted by $\B$ are taken to be open, whilst those denoted by $\overline{\B}$ are taken to be closed.

Here we obtain a few lemmata recording how, and in what sense, various almost-hyperbolic conditions propagate forwards in time under Ricci flow.
The first result establishes that if a flow $g(t)$ is initially locally almost-hyperbolic, then by reducing to a controllably smaller spatial region, 
the rescaled flow $\frac{g(t)}{1+2t}$ must remain close to being hyperbolic in a continuous sense. The precise result is the following.

\begin{lemma}[Barriers for rescaled flow]
\label{Barriers}
There is a universal constant $\ep > 0$  such that given any $b \in \left(0 , \frac{1}{2}\right]$ there exists a constant $J = J(b) > 0$ for which the following holds: 

Assume that $R \geq J$ and that $(\m , \cH )$ is a smooth surface. Suppose that for some $x \in \m$ we have both $\B_\cH (x,R) \subset \subset \m$ and that
$\left( \B_\cH (x,R) , \cH \right)$ is isometric to a hyperbolic disc of radius $R.$
Suppose $g(t)$ is a smooth Ricci flow defined on $\m$ for all $t \in [0,T],$ for some $T > 0,$ with $g(0)$ conformal to $\cH,$
and satisfying that for any $z \in \B_{\cH} (x,R)$ and $t \in [0, T]$ we have $ \B_{g(t)} (z,1) \subset \subset \m.$ 
Further suppose that 
\beq
	\label{AB} 
		(i) \quad (1 - b) \cH \leq  g(0) \leq (1 + b) \cH \qquad \text{and} \qquad (ii) \quad \left| \K_{g(0)} \right| \leq 2
\eeq
throughout $\B_\cH (x , R).$ Let $\tau := \min \{ \ep , T \} > 0.$ Then we may conclude that 
\beq
	\label{Barriers_preserved} 
		(1 - b) \cH \leq \frac{g(t)}{1 + 2t} \leq ( 1 + b ) \cH
\eeq
throughout $\ovB_\cH (x , R - J) \times [0, \tau].$
\end{lemma}
\vskip 4pt
\noindent
Observe that $\cH_{\pm} (t) := (1 \pm b + 2t) \cH$ are both Ricci flows with $\cH_{+}(0) = (1+b)\cH$ and $\cH_{-}(0) = (1-b)\cH.$
Since $(1-b) \cH < \frac{\cH_{\pm}(t)}{1+2t} < (1+b) \cH$ for positive times $t > 0,$ it is reasonable to expect that 
on a smaller spatial region $g(t)$ should remain sandwiched as in \eqref{Barriers_preserved} for a definite amount of time.

As we will see in the proof, the Gauss curvature bound assumed in $(ii)$ of \eqref{AB} means that Theorem \ref{Chenpseudo} allows us to conclude that $(1-b)e^{-8t} \cH \leq g(t) \leq (1+b)e^{8t} \cH$
throughout $\ovB_\cH (x , R - 2) \times [0,\ep]$ for a universal $\ep > 0.$
By restricting $\ep$ to being sufficiently small, we see that this almost establishes \eqref{Barriers_preserved} in that we can deduce that $\frac{g(t)}{1+9t} \leq (1+b) \cH$
and $\frac{g(t)}{1-9t} \geq (1-b) \cH.$
The content of the lemma is to establish that we may replace $1+9t$ and $1-9t$ by the \emph{same} function $1+2t$ and still preserve the barriers 
for a universal time $\ep > 0.$

Our strategy is to pullback our flow $g(t)$ to $\D$ 
so that, if we still call the pulled back flow $g(t)$, we have the barriers 
$(1-b)e^{-8\ep} h \leq g(t) \leq (1+b)e^{8\ep}h$ 
throughout $\ovB_h(0,R-2) \times [0,\ep]$.
We can reduce to only needing to consider the origin $0$, 
instead of an arbitrary $w \in \ovB_h(0,R-J-2)$, by considering a 
suitable M\"{o}bius map $\D \to \D$
(cf. Section \ref{2d_advantage}).
Hence we need only prove that having these barriers on $\B_h(0,J)$ 
allows us to establish the barriers in 
\eqref{Barriers_preserved} at the origin.

We use the barriers from above to choose dilations 
$h_{\al}$ and $h_{\mu}$ of $h$ for which 
$(1-b)h_{\mu} \leq (1-b)e^{-8\ep} h$ on $\partial \B_h (0, J)$ and
$1+b)e^{8\ep}h \leq (1+b)h_{\al}$ on $\partial \B_h(0,J)$. 
We will also ensure that 
$h_{\mu} \leq h \leq h_{\al}$ so that 
$(1-b)h_{\mu} \leq g(0) \leq (1+b)h_{\al}$ throughout
the entirety of $\B_h (0,J)$.
These orderings will mean that the Ricci flow
taking $(1+b)h_{\al}$ as its initial metric remains an upper barrier
for $g(t)$ throughout the parabolic boundary of
$\B_{h}(0,J) \times [0,\ep]$,
whilst the Ricci flow taking $(1-b)h_{\mu}$ as its initial metric
is a lower barrier for $g(t)$ throughout the parabolic
boundary of $\B_h (0, J) \times [0,\ep]$.
Consequently, the comparison principle (Theorem \ref{comp princ})
yields that these flows are in fact barriers for 
$g(t)$ throughout the entirety of $\B_h(0,J) \times [0,\ep]$.
The result then follow from showing that, provided $J$ is 
chosen sufficiently large, the dilations $h_{\al}$ and $h_{\mu}$
can be chosen to
provide the improved control at the origin required to conclude 
the improved barriers for $\frac{g(t)}{1+2t}$ 
claimed in \eqref{Barriers_preserved}. 

\begin{proof}[Proof of Lemma \ref{Barriers}]
Let $h$ denote the complete conformal hyperbolic metric of constant Gauss curvature $-1$ on $\D.$
Observe that $\VolBB_h (z , r) \geq \pi r^2$ for all points $z \in \D$ and any radius $r \in (0,1].$  
Let $\ep > 0$ be the universal constant arising from appealing to
the pseudolocality result of Chen, Theorem \ref{Chenpseudo}, with $r_0$ and $v_0$ there equal to $\frac{1}{\sqrt{2}}$ and $\frac{\pi}{4}$ respectively.
In particular, this tells us that if $ ( M^2 , g(t))$ is a smooth Ricci flow defined for $t \in [0,T],$ where $T > 0$ is arbitrary, and if $y \in M$ such that $\B_{g(t)} \left( y , \frac{1}{\sqrt{2}} \right) \subset \subset M$ for all $t \in [0,T],$ 
$| \K_{g(0)}| \leq 2$ throughout $\B_{g(0)} \left( y , \frac{1}{\sqrt{2}} \right)$ and $\VolBB_{g(0)} \left( y , \frac{1}{\sqrt{2}} \right) \geq \frac{\pi}{8},$
then $|\K_{g(t)}(y)| \leq 4$ for all $t \in [0, \tau],$ where $\tau := \min \{ \ep , T \} > 0.$ 
We fix this universal $\ep > 0$ for the remainder of the proof.

Given $b \in \left(0 , \frac{1}{2} \right]$ we seek to specify a constant $J= J(b) > 0$ so that, on a closed $\cH$ ball of radius $R-J,$ 
the barriers in $(i)$ of \eqref{AB} are valid for positive times for the rescaled family $\frac{g(t)}{1+2t}.$
With the benefit of hindsight, it will suffice to take
\beq
	\label{def_of_J}
		J(b) := 2+ \frac{1}{b}\max \left\{ 4e^{10 \ep} , 12 \right\} > 2.
\eeq
After locally pulling back to the disc $\D,$  it will be convenient to work with the Euclidean distance.  
Recall from Section \ref{2d_advantage} that a $h$ ball of radius $r$ centred at $0 \in \D$ corresponds to a Euclidean ball of radius $\tanh (r/2)$ centred at $0.$
Later in the proof we will end up working on a $h$ ball of radius $J -2$ centred at the origin $0 \in \D,$ which corresponds to $\D_j$ where $j := \tanh ( (J-2)/2 ).$ 
For use later we record that the bounds in \eqref{def_of_J} give that
\beq
	\label{euc_dist_ests}
		j := \tanh \left(\frac{J-2}{2}\right) \geq \max \left\{ 1- \frac{b}{2}e^{-10\ep} , 1 - \frac{b}{6} \right\} > 0
\eeq
via the inequality $\tanh (y) \geq 1 - \frac{1}{y}$ for $y >0$ 
(cf. Lemma \ref{tanh_bound}).

With both $\ep > 0$ and $J > 0$ specified, we let $R \geq J,$ $T > 0$ and define $\tau := \min \{ \ep , T \} > 0.$
Assume that $g(t)$ is a smooth Ricci flow on $\m,$
defined for all $t \in [0, T],$ with $g(0)$ conformal to $\cH,$ and satisfying that for every $z \in \B_\cH (x,R)$ and every $t \in [0,T]$ we have $\B_{g(t)} (z , 1) \subset \subset \m.$ 
Further suppose $g(0)$ satisfies both estimates $(i)$ and $(ii)$ in \eqref{AB} throughout $\B_\cH(x,R).$

Since $R \geq J > 2$ we may consider $z_0 \in \B_{\cH} (x , R - 3/2)$ so that $\B_{\cH} ( z_0 , 1 ) \subset \subset \B_{\cH} (x, R).$
Moreover, the barrier estimates $(i)$ of \eqref{AB} ensure that
\beq
	\label{contains} 
		\B_{\cH} \left( z_0 , \frac{1}{2}\right) \subset \B_{g(0)} \left(z_0 , \frac{\sqrt{3}}{2\sqrt{2}}\right) \subset \B_{g(0)} \left(z_0 , \frac{1}{\sqrt{2}}\right)  \subset \B_{\cH} (z_0 , 1) \subset \subset \B_{\cH} (x , R).
\eeq
The inclusions of \eqref{contains} allow us to simultaneously conclude that $| \K_{g(0)}| \leq 2$ throughout $\B_{g(0)} \left( z_0 , \frac{1}{\sqrt{2}}\right)$ via $(ii)$ of \eqref{AB},
and that 
$\VolBB_{g(0)}  \left( z_0 , \frac{1}{\sqrt{2}} \right)  \geq \frac{\pi}{8}.$ 
Recalling how $\ep > 0$ was chosen, Theorem \ref{Chenpseudo} tells us that $| \K_{g(t)} (z_0) | \leq 4$ for all $t \in [0,\tau].$
Repeating for all such points $z_0$ allows us to conclude that $| \K_{g(t)}| \leq 4$ throughout $\B_{\cH} (x , R-3/2) \times [0,\tau].$
Recalling \eqref{'2dric}, estimate $(i)$ in \eqref{AB} and the Gauss curvature control allows us to conclude that
$(1-b)e^{-8\ep} \cH \leq g(t) \leq (1+b)e^{8\ep} \cH$
throughout $\B_{\cH} (x , R-3/2) \times [0,\tau].$

To establish that $(1-b) \cH \leq \frac{g(t)}{1+2t} \leq (1+b) \cH$ throughout $\ovB_{\cH} ( x , R-J) \times [0,\tau]$ we pull back to $\B_h (0,R) \subset \D.$ 
That is, we pull back via the isometry $F:(\B_h (0, R) , h) \to ( \B_\cH (x , R) , \cH).$ 
After doing so we have a smooth Ricci flow $F^{\ast} g(t)$ defined on $\B_h (0, R)$ for all $t \in [0,T],$ and in particular satisfying that
$(1-b)h \leq F^{\ast}g(0) \leq (1+b)h$ throughout $\B_h(0,R)$ and
$(1-b)e^{-8\ep}h \leq F^{\ast}g(t) \leq (1+b)e^{8\ep}h$ throughout $\B_h (0, R-3/2)  \times [0,\tau].$ 
If we can establish that $(1-b)h \leq \frac{F^{\ast}g(t)}{1+2t} \leq (1+b)h$ throughout $\ovB_h (0, R-J) \times [0,\tau]$ then
the isometry will allow us to conclude \eqref{Barriers_preserved} as required.

Given any $w \in \ovB_h (0, R-J) \subset \D$ we can choose a M\"{o}bius diffeomorphism $\D \to \D$ mapping the origin $0$ to $w.$ 
Recalling from Section \ref{2d_advantage} that the pointwise difference between any metric and the hyperbolic metric $h$ are preserved under pulling back via M\"{o}bius diffeomorphisms,
establishing the following claim is sufficient to complete the proof.
\begin{claim}
Suppose $g(t)$ is a smooth Ricci flow on $\B_h (0, J - 3/2),$ defined for all $t \in [0,\tau],$ and satisfying both 
$(1-b)h \leq g(0) \leq (1+b)h$ throughout $\B_h (0, J - 3/2)$ and $(1-b)e^{-8\ep}h \leq g(t) \leq (1+b)e^{8\ep} h$ throughout $\B_h (0, J - 3/2) \times [0,\tau].$
Then at the origin $0 \in \D$ we have $(1-b)h \leq \frac{g(t)}{1+2t} \leq (1+b)h$ for all $t \in [0,\tau].$
\end{claim}

\begin{claimproof}
Let $j_0 := \tanh \left(\frac{J - \frac{3}{2}}{2}\right)$ and recall that $j = \tanh \left(\frac{J-2}{2}\right)$ so that $\B_h (0, J -2) = \D_j \subset \subset \D_{j_0} = \B_h (0, J - 3/2).$
Let $u : \D_{j_0} \times [0,\tau] \to \R$ be the smooth scalar function for which $g(t) = e^{2u} |dz|^2.$
In particular, we have that $u \in C^{\infty} \left( \overline{\D_j} \times [0,\tau] \right).$
Recalling that $h = e^{2\vph}|dz|^2,$ where $\vph(z) = \log \left[ \frac{2}{1-|z|^2}\right],$ the barriers $(1-b)h \leq g(0) \leq (1+b)h$ and $(1-b)e^{-8\ep} h \leq g(t) \leq (1+b)e^{8\ep}h$
become
\beq
	\label{K0}
		\frac{1}{2} \log (1-b) \leq u(z,0) - \vph(z) \leq \frac{1}{2} \log (1+b) 
\eeq
for $z \in \D_{j_0},$ and 
\beq
	\label{K1} 
		\frac{1}{2} \log (1-b) -4 \ep \leq  u( z ,t) - \vph (z)   \leq  \frac{1}{2} \log (1+b) + 4 \ep
\eeq 
for $ (z,t) \in \D_{j_0} \times [0,\tau]$ respectively.

We now define suitable Ricci flows between which our flow $g(t)$ will remain sandwiched. 
The upper barrier will follow from considering a complete Ricci flow $h_{\al}(t)$ on the disc of radius $\al = \al(j) \in (j,1)$ with initial Gaussian curvature $-(1+b)^{-1} \al^{-2}$
where $\al$ is taken to be $\al (j) := \left( \frac{ e^{4\ep} j^2}{e^{4\ep} + j^2 - 1} \right)^{\frac{1}{2}}.$
By observing that $\al(s)$ is strictly increasing as a function of $s$ and that $\al(0)=0$ and $\al (1) = 1$ we see that $\al(j) \in (0,1).$
A simple computation verifies that $\al(j) > j$ as required.
The conformal factor of this flow may be written as 
\beq
	\label{upper_barrier_conf_fac}
		H_{\al} (z,t) := \vph_{\al}(z) + \frac{1}{2} \log (1+b) + \frac{1}{2} \log \left( 1 + \frac{2t}{(1+b) \al^2} \right)
\eeq
where $\vph_{\al} (z) := \vph \left( \frac{z}{\al} \right)$ so that $\vph \leq \vph_{\al}$ where both defined.
In particular, one can compute from the definition of $\al$ that if $|z| = j$ then $\vph_{\al} (z) = \vph(z) + 4 \ep$
(having ensured $\al > j$ means that $\vph_{\al}$ is defined for $|z|=j$).

As a function, $H_{\al} \in C^{\infty} ( \D_{\al} \times [0,\infty) )$ thus, in particular, smooth on $\overline{\D_j} \times [0,\tau]$ since $\D_j \subset \subset \D_{\al}.$
Moreover, recalling \eqref{K0}, we see that \eqref{upper_barrier_conf_fac} ensures that $H_{\al} (z , 0) \geq u(z,0)$ throughout $\D_j,$
whilst for $(z,t) \in \partial \D_j \times [0,\tau]$ we may compute, using \eqref{K1}, that
$H_{\al} ( z , t) \geq \vph_{\al} ( z) + \frac{1}{2} \log( 1 + b) = \varphi ( z) + 4\ep + \frac{1}{2}\log ( 1 + b) \geq u ( z , t )$
since $z \in \partial \D_j$ means $|z|=j.$ 

We are now in a position to apply the variant of the comparison principle stated in Theorem \ref{comp princ} to deduce that $ H_{\al} \geq u$ throughout $\overline{ \D_j} \times [0,\tau].$
Since at the origin $0 \in \D_j$ we have $\vph_{\al}(0) = \vph (0),$ we see that at the origin $H_{\al} \geq u$ is equivalent to 
\beq
	\label{upper_barrier_1}
		g(t) \leq \left( 1 + b + \frac{2t}{\al^2} \right) h.
\eeq 
The lower barrier is constructed in a similar fashion. This time we consider a complete Ricci flow $h_{\mu} (t)$ on the disc of radius $\mu = \mu (j) > 1$ with Gaussian curvature initially $-(1-b)^{-1} \mu^{-2}$
where $\mu$ is taken to be $\mu (j) := j \left( 1 - (1-j^2) \exp \left[ \frac{5-4b}{1-b} \ep \right]\right)^{-\frac{1}{2}}.$
For this to make sense we require $1 - (1-j^2) \exp \left[\frac{5-4b}{1-b} \ep \right] > 0,$ which will be the case provided $1- e^{-10 \ep} < j^2.$
From \eqref{euc_dist_ests} we know that $j \geq 1 - \frac{b}{2}e^{-10\ep}$ and so, via Bernoulli's inequality, $j^2 > 1 - be^{-10 \ep}$ which is a little stronger than required.
A straightforward computation shows that $\mu(j) > 1$ as claimed.

The restriction of this flow to $\D_j$ yields a (now incomplete) flow which acts as a lower barrier for our flow $g(t)$ on $\D_j.$
To see this observe that the conformal factor of this flow can be written as
\beq
	\label{lower_barrier_conf_fac}
		H_{\mu} (z,t) := \vph_{\mu} (z) + \frac{1}{2} \log (1-b) + \frac{1}{2} \log \left( 1 + \frac{2t}{(1-b)\mu^2}\right)
\eeq 
where $\vph_{\mu}(z) := \vph \left( \frac{z}{\mu}\right)$ so that $\vph_{\mu} \leq \vph$ where both defined. 
As a function $H_{\mu} \in C^{\infty} ( \D \times [0,\infty))$ and thus, in particular, smooth on $\overline{\D_j} \times [0,\tau].$ 
Moreover, recalling \eqref{K0}, we see that \eqref{lower_barrier_conf_fac} ensures that $ H_{\mu} (z,0) \leq u(z,0)$ throughout $\D_j.$
Further, if $z \in \partial \D_j$ then $|z|=j$ and so $\vph_{\mu} (z) = \vph(z) -4 \ep - \frac{\ep}{1-b}.$
Therefore we may deduce that
\beq
	\label{getting_there}
		\vph_{\mu} (z) + \frac{1}{2} \log \left( 1 + \frac{2t}{(1-b) \mu^2} \right) \leq \vph_{\mu}(z) + \frac{t}{(1-b) \mu^2} \leq \vph_{\mu}(z) + \frac{\ep}{(1-b)} \leq \vph(z) - 4 \ep
\eeq
for all $(z,t) \in \partial \D_j \times [0,\tau]$ where we have used the inequality $\log x \leq x - 1.$
Hence \eqref{K1} and \eqref{getting_there} allows us to conclude that $ H_{\mu} \leq u$ throughout $\partial \D_j \times [0,\tau].$

We are now in a position to apply the variant of the comparison principle stated in Theorem \ref{comp princ} to deduce that $ H_{\mu} \leq u$ throughout $\overline{ \D_j} \times [0,\tau].$
Since at the origin $0 \in \D_j$ we have $\vph_{\mu}(0) = \vph (0),$ we see that at the origin $H_{\mu} \leq u$ is equivalent to 
\beq
	\label{lower_barrier_1}
		\left( 1 - b + \frac{2t}{\mu^2} \right) h \leq g(t).
\eeq 
Combining \eqref{upper_barrier_1} and \eqref{lower_barrier_1} yields that
\beq
	\label{barriers_at_0}
		(1-b) \left( \frac{ 1 + \frac{2t}{(1-b)\mu^2}}{1+2t} \right)h  \leq \frac{g(t)}{1+2t} \leq  (1+b) \left( \frac{ 1 + \frac{2t}{(1+b)\al^2}}{1+2t} \right)h
\eeq
at the origin $0 \in \D$ for all times $t \in [0,\tau].$
The estimates of \eqref{barriers_at_0} yield the barriers required by the claim provided we have both 
\beq
	\label{required_bounds}
		(\bA) \quad \al^2 \geq \frac{1}{1+b} \qquad \text{and} \qquad (\bB) \quad \mu^2 \leq \frac{1}{1-b}.
\eeq
The estimate $(\bA)$ in \eqref{required_bounds} is true provided $$j^2 \geq \frac{ e^{4\ep} - 1} { e^{4\ep} - 1 + be^{4\ep}} = 1 - \frac{b}{1+b - e^{-4\ep}}.$$
From \eqref{euc_dist_ests} we know that $ j \geq 1 - \frac{b}{6}$ and thus $ j^2 \geq 1 - \frac{b}{3}$ via the Bernoulli inequality. 
This is a little stronger than required and hence $(\bA)$ in \eqref{required_bounds} is true.
The estimate $(\bB)$ in \eqref{required_bounds} is true provided 
$$ j^2 \geq \frac{ \exp \left[\frac{5-4b}{1-b} \ep\right] - 1 }{ \exp \left[ \frac{5-4b}{1-b} \ep \right] - 1 + b} = 1 - \frac{b}{\exp \left[ \frac{5-4b}{1-b}\ep\right] -1 +b}. $$
From \eqref{euc_dist_ests} we know that $j \geq 1 - \frac{b}{2}e^{-10\ep}$ and thus $j^2 \geq 1 - b e^{-10 \ep}$ via the Bernoulli inequality. 
This is stronger than required and hence $(\bB)$ in \eqref{required_bounds} is true.
The estimates $(\bA)$ and $(\bB)$ in \eqref{required_bounds} combine with \eqref{barriers_at_0} to yield that
$(1-b)h \leq \frac{g(t)}{1+2t} \leq (1+b)h$
at the origin $0 \in \D$ for all $t \in [0,\tau],$ thus completing the proof of the claim.
\end{claimproof}
\vskip 2pt
\noindent
Combined with suitable M\"{o}bius diffeomorphisms, the claim allows us to establish the desired barriers for the pulled back flow $F^{\ast} g(t)$ on $\ovB_h (0,R-J) \times [0,\tau].$ 
The barriers in \eqref{Barriers_preserved} on $\ovB_{\cH} (x , R - J) \times [0,\tau]$ are then immediate by pulling back via the diffeomorphism $F^{-1}.$
This completes the proof of Lemma \ref{Barriers}.
\end{proof}
\vskip 4pt
\noindent
It is well known that $L^{\infty}$-barriers give rise to uniform 
$C^l$-estimates at strictly positive times. 
The following result uses this to establish Gauss curvature control away from time $0.$

\begin{lemma}[Barriers give curvature control] 
\label{curvature_control}
Let  $\al \in (0,1]$ and $S > 0.$ Then for any $\de \in (0, S)$ there exists a constant $b = b ( S , \al , \de ) > 0$
for which the following is true.

Assume that $(\m , \cH)$ is a smooth surface such that for some $x \in \m$ and $R \geq 2$ we have $\B_\cH (x,R) \subset \subset \m$ and that
$\left( \B_\cH (x,R) , \cH \right)$ is isometric to a hyperbolic disc of radius $R.$
Suppose that $g(t)$ is a smooth Ricci flow on $\m,$
defined for all $t \in [0, T]$ for some $T \in (0 , S],$ with $g(0)$ conformal to $\cH,$ and we have the barriers 
\beq
	\label{L-inf}
		(1-b)\cH \leq \frac{g(t)}{1+2t} \leq (1+b)\cH
\eeq
throughout $\B_\cH (x, R) \times [0,T].$ Then we may conclude that we have the Gauss curvature bounds
\beq 
	\label{Gauss-cont}
		-1 - \al \leq \K_{\frac{g(t)}{1+2t}} \leq -1 + \al
\eeq
throughout $\ovB_\cH ( x , R - 2) \times [\de,T].$ The estimates of \eqref{Gauss-cont} are vacuous if $T < \de.$
\end{lemma}

\begin{proof}[Proof of Lemma \ref{curvature_control} (sketch)]
Since the assumption \eqref{L-inf} is preserved under the pull back by diffeomorphisms,
we can pull back via the isometry 
$F : ( \B_h (0,R) , h) \to ( \B_\cH (x , R) , \cH )$
to obtain a smooth Ricci flow defined throughout $\B_h (0,R) \times [0,T],$ and satisfying the barriers in \eqref{L-inf} throughout this region of space-time,
and, thanks to $g(0)$ being conformal to $\cH,$ is given by $\omega h$ throughout $\B_h (0,R) \times [0,T]$ for some smooth function
$\omega : \B_h (0, R) \times [0,T] \to \R.$
Establishing the Gauss curvature estimates required in \eqref{Gauss-cont} for the pulled back flow will allow us to instantly deduce the required Gauss curvature
estimates for the flow $g(t)$ itself by pulling back via $F^{-1}.$

Thus we are reduced to needing to establish that if $g(t)$ is a smooth Ricci flow on $\B_h (0, R),$ defined for all $t \in [0,T],$ 
satisfying the barriers of \eqref{L-inf} throughout $\B_h (0,R) \times [0,T],$ and with $g(t) = \omega h$ for a smooth function $w : \B_h (0, R) \times [0,T] \to \R,$
then $g(t)$ satisfies the Gauss curvature estimates of \eqref{Gauss-cont} throughout $B_h (0, R-2) \times [\de, T].$
By utilising M\"{o}bius diffeomorphisms mapping $0$ to arbitrary $w \in \ovB_h ( 0, R -2),$ we may further reduce to only needing to establish the case $R=2.$ 
That is, having the barriers in \eqref{L-inf} throughout $\B_h (0,2) \times [0,T]$ yields the estimates in \eqref{Gauss-cont} at the origin $0$
for all times $t \in [\de , T].$

Whilst we have not yet specified our constant $b > 0,$ we may impose that we will require $b \in (0, 1/2],$ say. Therefore, 
the barriers in \eqref{L-inf} provide $L^{\infty}$-estimates on the conformal factor $u$ (for which $g(t) = e^{2u} |dz|^2$) throughout $\B_h(0,2) \times [0,T]$ depending only on $S.$
Since $g(t)$ is a Ricci flow, recalling \eqref{2dric}, the conformal factor $u$ satisfies the quasi-linear PDE $\frac{\partial u}{\partial t} = e^{-2u} \Delta u.$
A standard application of quasilinear PDE regularity theory (in particular, Theorems V.I.I and IV.10.1 in \cite{Lady}) allows us to deduce $C^l$-estimates, 
with respect to the flat Euclidean metric $|dz|^2,$ over $\D_{1/4},$ for all times $t \in [\de , T],$ depending only on $S,$ $\de$ and $l.$

The required Gauss curvature control in \eqref{Gauss-cont} then follows via interpolation. That is, at any $t \in [ \de , T]$ we have 
$C^l$-estimates on the difference of the conformal factors of $\frac{g(t)}{1+2t}$ and $h,$ with respect to the flat Euclidean metric $|dz|^2,$ over $\D_{1/4}.$
These bounds allows us to interpolate between the $C^l$-estimates and the assumed $C^0$-estimates, using Lemma B.6 in \cite{Giesen3}, for example.
By doing so, we may obtain improved control on the Euclidean derivatives, up to second order, of the difference of the conformal factors at the origin. Lemma B.5 in \cite{Giesen3}
then allows us to control the hyperbolic derivatives, up to second order, of the difference of the conformal factors at the origin.
Directly computing the difference of the Gauss curvatures with respect to the conformal factors allows us to convert these derivative bounds into the required Gauss curvature estimates
of \eqref{Gauss-cont}, provided $b$ is sufficiently small, depending on $S,$ $\al$ and $\de$ only. 
The details of this outline are standard arguments, and may be found in \cite{McL}.
\end{proof}
\vskip 4pt
\noindent
In the case that we assume $g(0) \equiv \cH$ throughout $\m,$ we will
require a minor modification of Lemma \ref{curvature_control} to avoid any time delay before achieving our desired Gauss curvature control.
The result will exploit the uniform initial $C^l$-bounds provided by the initial equality.

\begin{lemma}[No time delay]
\label{no_time_delay}
Let $\al \in (0,1]$ and $S > 0.$ Then there exists a constant $b = b( S , \al ) > 0$ for which the following is true. 

Assume $(\m , \cH)$ is a smooth surface such that for some $x \in \m$ and $R \geq 2$ we have $\B_\cH (x,R) \subset \subset \m$ and 
$( \B_\cH (x,R) , \cH )$ is isometric to a hyperbolic disc of radius $R.$
Suppose $g(t)$ is a smooth Ricci flow on $\m,$ defined for all $t \in [0,T]$ for some $T \in (0,S],$ with $g(0) \equiv \cH$ throughout $\m,$ and
we have the barriers
\beq
	\label{no_delay_barriers}
		(1-b) \cH \leq \frac{g(t)}{1+2t} \leq (1+b) \cH
\eeq 
throughout $\B_\cH (x , R) \times [0,T].$
Then we may deduce that we have the Gauss curvature bounds
\beq
	\label{no_delay_gauss_bounds}
		-1 - \al \leq \K_{\frac{g(t)}{1+2t}} \leq -1 + \al
\eeq
throughout $ \ovB_\cH (x,R-2) \times  [0,T].$
\end{lemma}

\begin{proof}[Proof of Lemma \ref{no_time_delay}(sketch)]
The proof is almost identical to the proof of Lemma \ref{curvature_control}. 
The exact same reasoning as in the proof of Lemma \ref{curvature_control} allows us to reduce to working on $\D,$ and only needing to show 
that having the barriers in \eqref{no_delay_barriers} throughout $\B_h (0,2) \times [0,T]$ yields the estimates in \eqref{no_delay_gauss_bounds} at the origin $0.$
However, we now also have, after making the reduction to this case, that $g(0) \equiv h$ throughout $\B_h (0,2).$
Having $g(0) \equiv h$ throughout $\B_h (0,2)$ allows us to deduce uniform initial $C^l$-estimates for the conformal factor $u$ (for which $g(t) = e^{2u} |dz|^2$)
throughout $\B_h(0,2),$ whilst the barriers in \eqref{no_delay_barriers} still provide $L^{\infty}$-bounds throughout $\B_h (0,2) \times [0,T].$
These additional time $0$ uniform $C^l$-estimates allow us to appeal to quasilinear PDE regularity theory. Again we use Theorems V.I.I and IV.10.1 in \cite{Lady}, 
but now the variants that only require moving away from the spatial boundary, and hence yield $C^l$-estimates, with respect to the flat Euclidean metric $|dz|^2,$ over $\D_{1/4},$ for all times $t \in [0,T].$
With these estimates obtained, we proceed verbatim as in the proof of Lemma \ref{curvature_control}.
The details of this outline are again standard arguments, and may be found in \cite{McL}.
\end{proof}

\section{Improved Time Control}
\label{main_results_obtained}

We are now ready to complete the proof of both Theorems \ref{hypat0} 
and \ref{mainresult}. They are both consequences of the following
theorem.

\begin{theorem}
\label{over_arch_thm}
Let $ \al \in (0,1]$ be given. Then there is a universal constant $\ep > 0$ 
such that for any $\de \in (0, \ep)$ there exist constants $b = b( \al , \de) >0$ and $\Lambda = \Lambda (\al ,\de) >0$ 
for which the following is true.

Suppose that $R \geq \Lambda$ and that $(\m , \cH)$ is a smooth surface 
which satisfies for some $x \in \m$ that the ball $\B_\cH (x,R) \subset \subset \m$ and $\left( \B_\cH (x , R) , \cH \right)$ is isometric to a hyperbolic disc of radius $R.$
Assume $g(t)$ is a smooth Ricci flow defined on $\m$ for all $t \in [0,T]$ for some $T > 0,$ with $g(0)$ conformal to $\cH,$ 
and satisfying that for any $l \in \N_0,$ if $z \in \B_\cH (x,R- l \Lambda )$ and $t \in [0,T]$ then
$\B_{g(t)} \left(z,(1+2\ep)^{\frac{l}{2}} \right) \subset \subset \m.$ Further suppose that
\beq
	\label{over_arch_assumed_ests}
		(\bA) \quad (1-b) \cH \leq g(0) \leq (1+b) \cH \qquad \text{and} \qquad (\bB) \quad | \K_{g(0)} | \leq 2
\eeq
throughout $\B_\cH(x,R).$ 
Then we have that 
\beq
	\label{Gauss_curv_bound_conc}
		- 1- \al \leq \K_{\frac{g(t)}{1+2t}} \leq -1 + \al
\eeq
throughout $\ovB_\cH \left(x , R - \left\lfloor \frac{R}{\Lambda} \right\rfloor \Lambda \right) \times [ \de , \Tau_{\max}]$ where
\beq
	\label{time_conc}
		\Tau_{\max} = \min \left\{ T , \frac{\exp \left[  \left\lfloor \frac{R}{\Lambda} \right\rfloor \log (1+2\ep) \right] - 1}{2} \right\}.
\eeq 
Moreover, if in place of the estimates in \eqref{over_arch_assumed_ests} we had that $g(0) \equiv \cH$ throughout $\m,$ 
then we may in fact deduce the estimates of \eqref{Gauss_curv_bound_conc} throughout 
$\ovB_\cH \left(x , R - \left\lfloor \frac{R}{\Lambda} \right\rfloor \Lambda \right) \times [0 , \Tau_{\max}],$ where $\Tau_{\max}$ is as specified in \eqref{time_conc}.
\end{theorem}
\vskip 4pt
\noindent
To clarify, for $z \in \R$ we have 
$\left\lfloor z \right\rfloor := \max \{ m \in \Z : m \leq z \}.$
Before starting the proof we outline the strategy. 
First, we repeatedly apply Lemma \ref{Barriers} followed by Lemma \ref{curvature_control} to deduce that 
$(1-b)\cH \leq \frac{g(t)}{1+2t} \leq (1+b)\cH$ throughout 
$\ovB_\cH \left(x , R - \left\lfloor \frac{R}{\Lambda} \right\rfloor \Lambda + 2 \right) \times [0 , \Tau_{\max}],$
where $\Tau_{\max}$ is as specified in \eqref{time_conc}.
To elaborate, Lemma \ref{Barriers} establishes these barriers 
over some time interval $[0,\tau]$ for $\tau > 0.$
Lemma \ref{curvature_control} then provides Gauss curvature bounds
at time $\tau$ that allow us to apply Lemma \ref{Barriers}
again, but this time using the metric $\frac{g(\tau)}{1+2\tau}$
as the initial metric. We repeatedly apply these lemmas, in
this order, as many times as possible.

Successively applying these lemmas will require considering
suitable rescaled versions of the original flow $g(t)$. 
For example, to make the second application from time $\tau$ 
onwards will require considering a rescaled variant of $g(t)$
that takes $\frac{g(\tau)}{1+2\tau}$ as its initial metric.
A consequence of the rescaling is that each application will
obtain the barriers for $\frac{g(t)}{1+2t}$ on progressively 
longer time intervals.
Together with keeping track of the maximum number of times we 
can iterate this procedure, theses increases in the length of the
successive time intervals will lead to the explicit form of 
$\Tau_{max}$ in \eqref{time_conc}.

Having established the barriers, we will then be able to use 
Lemma \ref{curvature_control} to obtain the Gauss curvature control 
of \eqref{Gauss_curv_bound_conc} throughout
$\ovB_\cH \left(x , R - \left\lfloor \frac{R}{\Lambda} \right\rfloor \Lambda \right) \times [\de , \Tau_{\max}].$
Of course, our iterative procedure in the first step invoked Lemma 
\ref{curvature_control}, and so provided these Gauss curvature bounds
for certain times already. However, each time Lemma \ref{curvature_control}
was used, there was an initial portion of the time interval 
under consideration for which the Gauss curvature bounds are
\emph{not} established. A convenient way of overcoming this 
problem is to directly argue that the barriers established in the
first step give the Gauss curvature control 
of \eqref{Gauss_curv_bound_conc} for all times 
$t \in [\de , \Tau_{max}].$

To achieve this, fix $s \in [\de , \Tau_{max}]$ and consider $g(t)$
for $t \in [s - \de, s]$. We may then use the barriers from the first 
step to apply Lemma \ref{curvature_control} to a suitably rescaled
version of $g(t)$ on this time interval. Consequently, after 
rescaling back to the original flow, we have the Gauss curvature 
bounds in \eqref{Gauss_curv_bound_conc} throughout 
$\ovB_\cH \left(x , R - \left\lfloor \frac{R}{\Lambda} \right\rfloor \Lambda \right)$ at time $t=s$.
Repeating for each $s \in [\de , \Tau_{max}]$ then yields the desired
conclusion.

Finally, for the case that $g(0) \equiv \cH$ we make one further
step. Namely, we additionally apply Lemma \ref{no_time_delay}
to avoid any time delay before achieving the Gauss curvature control
of \eqref{Gauss_curv_bound_conc}.

\begin{proof}[Proof of Theorem \ref{over_arch_thm}]
Retrieve the universal constant $\ep > 0$ from Lemma \ref{Barriers}. Let $\al \in (0,1]$ and $\de \in (0, \ep)$ both be given. 
Retrieve the constant $ b_1 = b_1(\al , \de) > 0$ arising in Lemma \ref{curvature_control} for the $S,$ $\al$ and $\de$ there equal to $\ep,$ $\al$ and $\de$ here respectively.
With the aim of avoiding any time delay before obtaining the estimates of \eqref{Gauss_curv_bound_conc} in the case $g(0) \equiv \cH,$
retrieve the constant $b_2 = b_2 ( \al ) > 0$ arising in Lemma \ref{no_time_delay} for the $S$ and $\al$ there given by $\ep$ and $\al$ here respectively.
Take $b := \min \{ b_1 , b_2 \} > 0$ which depends only on $\al$ and $\de.$ By reducing $b$ if required, but without additional dependency, we may assume that $b \in (0, 1/2].$
This means we may define $\Lambda = \Lambda ( \al , \de ) := J(b) + 2 > 0$ where $J(b)$ is the constant arising in Lemma \ref{Barriers}.
We fix these quantities for the remainder of the proof.

We first deal with the case $T \in (0, \ep].$ That is, assume we are in the setting of the theorem with $T \in (0, \ep].$
The estimates on $g(0)$ in \eqref{over_arch_assumed_ests}, together with the assumed compact inclusions for $l=0$ and that $g(0)$ is conformal to $\cH,$
provide the required hypotheses to apply Lemma \ref{Barriers} to the flow $g(t).$
Doing so, and recalling that $\tau := \min \{ T , \ep \} = T \leq \ep,$ yields the barriers $(1-b)\cH \leq \frac{g(t)}{1+2t} \leq (1+b) \cH$
throughout $\ovB_\cH (x , R - \Lambda + 2 ) \times [0, T],$ recalling that $\Lambda = J(b) + 2 > 0$ where $J(b)$ is the constant arising in Lemma \ref{Barriers}.

In turn, these barriers are of the form required by Lemma \ref{curvature_control}. Recalling how $b$ was specified, we observe that we have the required hypothesis to apply
Lemma \ref{curvature_control} to $g(t)$ and deduce that $-1 - \al \leq \K_{\frac{g(t)}{1+2t}} \leq -1 + \al$ throughout
$\ovB_\cH (x , R - \Lambda) \times [\de , T].$ Of course, these Gauss curvature estimates are vacuous if $T < \de.$  
Since $R \geq \Lambda$ we see that $ \left\lfloor \frac{R}{\Lambda} \right\rfloor \geq 1,$ and so we have established the Gauss curvature estimates required in 
\eqref{Gauss_curv_bound_conc} throughout $\ovB_\cH \left( x , R - \left\lfloor \frac{R}{\Lambda} \right\rfloor \Lambda \right) \times [\de, T],$
which is for the time required in \eqref{time_conc}.

In the case that the estimates in \eqref{over_arch_assumed_ests} are replaced by the assumption that $g(0) \equiv \cH$ throughout $\m$ 
we may appeal to Lemma \ref{no_time_delay} in place of Lemma \ref{curvature_control}.
By doing so, we conclude that $-1 - \al \leq \K_{\frac{g(t)}{1+2t}} \leq -1 + \al$ throughout $\ovB_\cH (x , R - \Lambda) \times [0 , T].$ 
Again $R \geq \Lambda$ means that $ \left\lfloor \frac{R}{\Lambda} \right\rfloor \geq 1,$ and so we have established the Gauss curvature estimates required in 
\eqref{Gauss_curv_bound_conc} throughout $\ovB_\cH \left( x , R - \left\lfloor \frac{R}{\Lambda} \right\rfloor \Lambda \right) \times [0, T],$
giving the required improvement.

For the remainder of the proof we assume that $T > \ep.$ We proceed under the assumptions that $g(0)$ satisfies both the estimates specified in \eqref{over_arch_assumed_ests},
and will only later make a single extra step to remove the time delay before we obtain the estimates in \eqref{Gauss_curv_bound_conc} when we have the initial equality $g(0) \equiv \cH.$
Our first goal is to establish that the flow $g(t)$ satisfies the barriers $(1-b) \cH \leq \frac{g(t)}{1+2t} \leq (1+b) \cH$ throughout 
$\ovB_\cH \left(x , R - \left\lfloor \frac{R}{\Lambda} \right\rfloor \Lambda + 2 \right) \times [0 , \Tau_{\max}],$ where $\Tau_{\max}$ is as specified in \eqref{time_conc}.
To achieve this, we will inductively apply Lemma \ref{Barriers} followed by Lemma \ref{curvature_control} to rescalings of $g(t).$

To illustrate, note we have the required hypotheses to appeal to Lemma \ref{Barriers} and deduce, since $\min \{ T , \ep \} =\ep$ now,
that we have the barriers $(1-b)\cH \leq \frac{g(t)}{1+2t} \leq (1+b)\cH$ throughout $\ovB_\cH ( x , R - \Lambda + 2) \times [0 , \ep].$
These barriers allow us to apply Lemma \ref{curvature_control} to the flow $g(t)$ to obtain that 
$-1 - \al \leq \K_{\frac{g(t)}{1+2t}} \leq -1 + \al$ throughout $\ovB_\cH ( x , R - \Lambda) \times [ \de , \ep].$ 
Since $\al \in (0,1],$ these Gauss curvature estimates tell us that $\left| \K_{\frac{g(\ep)}{1+2\ep}} \right| \leq 2$
throughout $\ovB_\cH (x , R - \Lambda).$
Therefore the metric $\frac{g(\ep)}{1+2\ep}$ satisfies the same barriers and Gauss curvature bounds throughout $\B_\cH ( x , R - \Lambda)$
as those satisfied by $g(0)$ throughout $\B_\cH (x,R).$
Hence it is natural to try to apply Lemma \ref{Barriers} to a rescaling of the flow $g(t)$ which takes $\frac{g(\ep)}{1+2\ep}$ as its initial metric. 

The rescaled Ricci flow $\tilde{g}(s)$ given by $\tilde{g}(s) := \frac{g( \ep + (1+2\ep)s)}{1+2\ep},$ defined on $\m$ for all $s \in \left[ 0 , \frac{T-\ep}{1+2\ep} \right],$
satisfies that $\tilde{g}(0) = \frac{g(\ep)}{1+2\ep}$ as required. Thus it is to this flow that we aim to apply first Lemma \ref{Barriers}, and then Lemma \ref{curvature_control}. 
Modulo checking that all of the required hypotheses are satisfied (which we will later do rigorously),  
the relationship between $\ep$ and $\frac{T-\ep}{1+2\ep}$ will determine whether this subsequent application of Lemmas \ref{Barriers} and \ref{curvature_control} establishes control
up until time $T,$ or if the flow $\tilde{g}(s)$ exists beyond $s = \ep,$ which itself corresponds to having $T > \ep + (1+2\ep)\ep.$

We also need to consider how the spatial region is changing. Each time we appeal to Lemma \ref{Barriers}, followed by Lemma \ref{curvature_control},
we require being able to move in to a spatial $\cH$ ball, centred at $x,$ of radius $\Lambda$ less than the original radius.
Therefore we can only make this application of Lemma \ref{Barriers}, followed by Lemma \ref{curvature_control}, to the flow $\tilde{g}(s)$ if we have that $R- \Lambda \geq \Lambda,$ i.e. if $R - 2 \Lambda \geq 0.$
If both $T > \ep + (1+2\ep)\ep$ and $R - 2\Lambda \geq 0$ are true, we could apply the lemmas as specified above to control the Ricci flow $\tilde{g}(s)$ up until $s = \ep.$
The aim would then be to repeat this procedure by considering a rescaling of $\tilde{g}(s)$ taking $\frac{\tilde{g}(\ep)}{1+2\ep}$ as its initial metric.

In order to implement this iterative process we introduce some notation. We define
$q \in \N_0$ to be the value 
\beq
	\label{cutie}
		q := \max \left\{ l \in \N_0 : \sum_{k=0}^{l} \ep (1+2\ep)^k \leq T \right\},
\eeq
which is possible since we are assuming $T > \ep.$
Let $N := \min \left\{ q , \left\lfloor \frac{R}{\Lambda} \right\rfloor - 1 \right\}.$ 
We will later see that $N+1$ corresponds to the maximum number of times we may iteratively appeal first to Lemma \ref{Barriers}, followed by Lemma \ref{curvature_control},
to establish the required barriers over a time interval of size $\ep,$ and the Gauss curvature control at the later time $\ep.$
For now, we observe that we necessarily have that $R - (N+1) \Lambda \geq 0,$ hence $R - i \Lambda \geq 0$ for every $i \in \{ 1 , \ldots , N +1 \}.$

For notational convenience we set $g_0 (t) := g(t)$ for $t \in [0, \ep].$ 
and recall that we have established that $(1-b)\cH \leq \frac{g_0(t)}{1+2t} \leq (1+b)\cH$ throughout $\ovB_\cH ( x , R - \Lambda + 2) \times [0,\ep]$
and that $\left| \K_{\frac{g_0(\ep)}{1+2\ep}} \right| \leq 2$ throughout $\ovB_\cH ( x, R - \Lambda).$

For $i \in \{ 1 , \ldots , N + 1\}$ we define   
\beq
	\label{tau_l}
		\tau_i := \frac{T - \sum_{k=0}^{i-1} \ep ( 1 + 2\ep)^k }{ (1+2\ep)^i}
\eeq
which will correspond to the (rescaled) remaining existence time for the flow $g(t)$ after having made $i$ applications of Lemmas \ref{Barriers} and \ref{curvature_control}. 
Naturally this means that $\tau_i >\tau_{i+1}$ when both are defined, and further we claim that $\tau_i \geq \ep$ for every $i \in \{ 1 , \ldots , N \}.$
To see this observe that $q \geq N,$ and hence from \eqref{cutie} we know that $T \geq \sum_{k=0}^q \ep(1+2\ep)^k \geq \sum_{k=0}^N \ep (1+2\ep)^k.$
Therefore, if $i \in \{ 1 , \ldots , N \},$ we can compute, using \eqref{tau_l}, that
$$ \tau_i := \frac{T - \sum_{k=0}^{i-1} \ep ( 1 + 2\ep)^k }{ (1+2\ep)^i} \geq \frac{\sum_{k=0}^N \ep (1+2\ep)^k - \sum_{k=0}^{N-1} \ep ( 1 + 2\ep)^k }{ (1+2\ep)^N} = \ep$$
as required.
For $i \in \{ 1 , \ldots , N  \}$ we inductively define 
\beq
	\label{g_i}
		g_i (t) := \frac{g_{i-1}( \ep + (1+2\ep)t)}{1+2\ep}
\eeq
which is a smooth Ricci flow defined on $\m$ for all $t \in [0, \tau_i].$
Previously, we have seen that $g_1(t)$ is defined on $\m$ for all $t \in [0,\tau_1].$
Then observe, for $i \in \{1 , \ldots , N-1\},$ that if $g_i(t)$ is defined on $\m$ for all $[0,\tau_i],$ then from \eqref{g_i} we see that $g_{i+1}(t)$ 
is defined on $\m$ for all $t \in [ 0 , t_{\ast}]$ where $t_{\ast}$ satisfies that $\ep + (1+2\ep)t_{\ast} = \tau_i.$
Hence $t_{\ast} = \frac{ \tau_i - \ep}{1+2\ep} = \tau_{i+1}$ as required.

Our assumption that for any $z \in \B_\cH ( x , R - i \Lambda )$ and all $t \in [0,T]$ that we have $\B_{g(t)} \left( z , (1+2\ep)^{\frac{i}{2}} \right) \subset \subset \m$ 
tells us that for any $z \in \B_\cH (x, R-i\Lambda )$ and all $t \in [  0 , \tau_i ]$ we have 
\beq
	\label{cmpct_reqs}
		\B_{g_{i}(t)} ( z , 1) = \B_{g \left( \sum_{k=0}^{i-1} \ep (1+2\ep)^k  + (1+2\ep)^{i} t \right)} \left( z , (1+2\ep)^{\frac{i}{2}} \right) \subset \subset \m.
\eeq
Recall that we have established both that $(1-b)\cH \leq \frac{g_0(\ep)}{1+2\ep} \leq (1+b)\cH$ and $\left| \K_{\frac{g_0(\ep)}{1+2\ep}} \right| \leq 2$  
throughout $\ovB_\cH ( x, R - \Lambda).$
In terms of $g_1(t),$ these give that $(1-b)\cH \leq g_1(0) \leq (1+b)\cH$ and $| \K_{g_1(0)} | \leq 2$ throughout $\ovB_\cH ( x, R - \Lambda).$ 
These estimates, together with the compact inclusions in \eqref{cmpct_reqs} (for $i=1$), provide the required hypotheses to apply Lemma \ref{Barriers} to the flow $g_1(t).$ 

In fact, we may proceed inductively, with the following claim giving the inductive step. 
\begin{claim}[Inductive step]
Suppose $i \in \{ 1 , \ldots , N \}$ and we have both $(1-b)\cH \leq g_i(0) \leq (1+b)\cH$ and $| \K_{g_i(0)} | \leq 2$ throughout $\B_\cH ( x , R - i\Lambda ).$
Then we have that
\beq
	\label{i_barriers}
		(1-b) \cH \leq \frac{g_i(t)}{1+2t} \leq (1+b) \cH
\eeq
throughout $\ovB_\cH ( x , R- (i+1) \Lambda + 2) \times [0, \ep],$ and
\beq
	\label{i_gauss}
		-1 - \al \leq \K_{\frac{g_i(t)}{1+2t}} \leq -1 + \al
\eeq
throughout $\ovB_\cH ( x , R- (i+1) \Lambda) \times [\de, \ep].$
Since $\al \in (0,1],$ a particular consequence of \eqref{i_gauss} is that we have $\left| \K_{\frac{g_i(\ep)}{1+2\ep}} \right| \leq 2$ throughout $\ovB_\cH ( x , R - (i+1)\Lambda).$
\end{claim}

\begin{claimproof}
The assumptions in the claim, combined with the compact inclusions of \eqref{cmpct_reqs} for $i,$ along with noting that $g_i(0)$ is conformal to $\cH,$
provide the required hypothesis to apply Lemma \ref{Barriers} to the flow $g_i(t).$ Since $\tau_i \geq \ep$ we can deduce the barriers in \eqref{i_barriers}
over $\ovB_\cH ( x , R - (i+1)\Lambda + 2) \times [0,\ep]$ as required.
The barriers in \eqref{i_barriers}, along with noting that $0 < \de < \ep \leq \tau_i$ and $R - (i+1)\Lambda + 2 \geq 2,$ allow us to appeal to Lemma \ref{curvature_control} to deduce 
the Gauss curvature estimates \eqref{i_gauss} throughout $\ovB_\cH ( x , R- (i+1) \Lambda) \times [\de, \ep]$ as claimed.
\end{claimproof}
\vskip 4pt
\noindent
By appealing to the inductive step in the claim a total of $N$ times, observing that the conclusions of the claim for $i \in \{ 1 , \ldots , N-1\}$ provide the required hypothesis
in order to appeal to the claim for $i+1,$ we can deduce the barriers in \eqref{i_barriers} for every $i \in \{ 1 , \ldots , N \},$
along with already having established such barriers for $i=0.$
Recalling \eqref{g_i}, we can compute that for $i \in \{ 1 , \ldots , N\}$ and $s \in [0,\ep]$ we have
\beq
	\label{i-version}
		\frac{g_i(s)}{1+2s} = \frac{g\left( \sum_{k=0}^{i-1}\ep(1+2\ep)^k + (1+2\ep)^i s\right)}{(1+2\ep)^i(1+2s)} = \frac{g\left( \sum_{k=0}^{i-1}\ep(1+2\ep)^k + (1+2\ep)^i s\right)}{1  + 2 \left( \sum_{k=0}^{i-1}\ep(1+2\ep)^k + (1+2\ep)^i s\right)  }
\eeq
where we have used that $1 + 2 \ep \sum_{k=0}^{i-1} (1+2\ep)^k = (1+2\ep)^i.$
Thus \eqref{i_barriers} tells us that 
\beq
	\label{i-barriers-for-g}
		(1-b)\cH \leq \frac{g(t)}{1+2t} \leq (1+b)\cH
\eeq
throughout $\ovB_\cH ( x , R - (i+1) \Lambda + 2) \times \left[ \sum_{k=0}^{i-1}\ep(1+2\ep)^k , \sum_{k=0}^i \ep (1+2\ep)^k \right].$
Combining \eqref{i-barriers-for-g} for each $i \in \{ 1 , \ldots , N\},$ 
and recalling that we already know that $(1-b)\cH \leq \frac{g(t)}{1+2t} \leq (1+b)\cH$ throughout $\ovB_\cH (x, R - \Lambda + 2) \times [0,\ep],$ yields that 
\beq
	\label{combined_barriers}
		(1+b)\cH \leq \frac{g(t)}{1+2t} \leq (1+b)\cH
\eeq
throughout $\ovB_\cH (x , R - (N+1)  \Lambda +2) \times \left[0 , \sum_{k=0}^{N} \ep (1+2\ep)^k \right].$

We must now split into two cases depending on the value taken by $N.$ 
If $N = \left\lfloor \frac{R}{\Lambda} \right\rfloor - 1$ then we do not have sufficient spatial room left to appeal to the claim.
In this case we can compute that
$$ \sum_{k=0}^{N} \ep (1+2\ep)^k = \frac{1}{2} \left( \exp \left[ (N+1) \log (1+2\ep) \right] -1 \right), $$
and since $N =  \left\lfloor \frac{R}{\Lambda} \right\rfloor - 1$ we see that this gives the form of $\Tau_{max}$
as claimed in \eqref{time_conc}.
Hence we have established the barriers of \eqref{combined_barriers} throughout 
$\ovB_\cH \left( x , R - \left\lfloor \frac{R}{\Lambda} \right\rfloor \Lambda + 2\right) \times [ 0 , \Tau_{\max}].$

If $N < \left\lfloor \frac{R}{\Lambda} \right\rfloor - 1$ then we still have the spatial room required to appeal to the claim.
However, in this case we necessarily have that $N = q$ and so $\tau_{N+1} < \ep,$ hence we can only establish control up to time $\tau_{N+1}.$
Indeed, consider the rescaled Ricci flow 
\beq
	\label{g_N}
		g_{N+1} (t) := \frac{g_{N}( \ep + (1+2\ep)t)}{1+2\ep}
\eeq
defined on $\m$ for all $t \in [0, \tau_{N+1}],$ where $g_{N}(t)$ is as defined in \eqref{g_i} for $i = N.$

Since we were able to apply the inductive step, as stated in the previous claim, to the flow $g_N(t),$
we know that we have both $(1-b)\cH \leq \frac{g_N(\ep)}{1+2\ep} \leq (1+b)\cH$ and $\left| \K_{\frac{g_N(\ep)}{1+2\ep} } \right| \leq 2$
throughout $\ovB_\cH (x , R - (N+1)\Lambda).$
Therefore, from \eqref{g_N} we see that these estimates tell us that we have both 
$(1-b)\cH \leq g_{N+1}(0) \leq (1+b)\cH$ and $\left| \K_{g_{N+1}(0) } \right| \leq 2$
throughout $\ovB_\cH (x , R - (N+1)\Lambda).$
Hence the compact inclusions in \eqref{cmpct_reqs} for $i=N+1,$ and the fact that $g_{N+1}(0)$ is conformal to $\cH,$ combine with the above estimates
to provide the required hypotheses to apply Lemma \ref{Barriers} to the flow $g_{N+1}(t).$
Doing so yields, recalling that $\tau_{N+1} < \ep,$ that 
\beq
	\label{final_step_barriers}
		(1-b) \cH \leq \frac{g_{N+1}(t)}{1+2t} \leq (1+b)\cH
\eeq
throughout $\ovB_\cH ( x , R - (N+2) \Lambda + 2) \times [0, \tau_{N+1}].$
Repeating the computations in \eqref{i-version} and \eqref{i-barriers-for-g} for $i=N+1$ we see that \eqref{final_step_barriers} yields that
\beq
	\label{final_step_barriers_for_g}
		(1-b) \cH \leq \frac{g(t)}{1+2t} \leq (1+b)\cH
\eeq
throughout $\ovB_\cH ( x , R - (N+2)\Lambda + 2  ) \times \left[ \sum_{k=0}^{N} \ep (1+2\ep)^k , \sum_{k=0}^{N} \ep (1+2\ep)^k + (1+2\ep)^{N+1} \tau_{N+1} \right].$
From \eqref{tau_l} we can compute that
$$ \sum_{k=0}^{N} \ep (1+2\ep)^k + (1+2\ep)^{N+1} \tau_{N+1} = T,$$
and since $N < \left\lfloor \frac{R}{\Lambda} \right\rfloor -1$ we must have that $R - (N+2)\Lambda +2 \geq R-  \left\lfloor \frac{R}{\Lambda} \right\rfloor \Lambda +2.$
These observations allow us to combine \eqref{combined_barriers} with \eqref{final_step_barriers_for_g} to deduce that
$	(1-b) \cH \leq \frac{g(t)}{1+2t} \leq (1+b)\cH $
throughout $\ovB_\cH \left( x , R-  \left\lfloor \frac{R}{\Lambda} \right\rfloor \Lambda +2 \right) \times [ 0 , T].$
Since $\Tau_{\max} \leq T,$ we have these barriers for all times $t \in [ 0 , \Tau_{\max}].$

In either case we have established that
\beq
	\label{good_barriers}
		(1-b)\cH \leq \frac{g(t)}{1+2t} \leq (1+b)\cH
\eeq
throughout $\ovB_\cH \left( x , R-  \left\lfloor \frac{R}{\Lambda} \right\rfloor \Lambda +2 \right) \times [ 0 , \Tau_{\max}].$
We will now use these barriers and Lemma \ref{curvature_control} to establish the Gauss curvature estimates required in \eqref{Gauss_curv_bound_conc}
throughout $\ovB_\cH \left( x , R - \left\lfloor \frac{R}{\Lambda} \right\rfloor \Lambda \right) \times [\de, \Tau_{\max}].$
Consider any $ s \in [ \de , \Tau_{\max}]$ 
and define $\gamma_s := \frac{s - \de} { 1 + 2\de} \in [0 , s).$
Then consider the Ricci flow $g_s (t) := \frac{g( \gamma_s + (1+2\gamma_s)t)}{1+2\gamma_s}$ on $\m,$ 
defined for all times $t \in \left[0, \frac{\Tau_{\max} - \gamma_s }{1+2\gamma_s}\right],$
and with $g_s (0)$ conformal to $\cH.$ 
Observe that 
\begin{align*} 
	\frac{\Tau_{\max} - \gamma_s }{1+2\gamma_s} - \de &= \frac{ \Tau_{\max} - \gamma_s - \de - 2 \gamma_s \de}{1+2 \gamma_s} \\
			&= \frac{ (1+2\de)\Tau_{\max} - (s-\de) - (1+2\de)\de - 2 (s-\de) \de}{(1+2 \gamma_s)(1+2\de)} \\
			&= \frac{1+2\de}{1+2s} ( \Tau_{\max} - s ) \geq 0
\end{align*}
where we have used that $1 + 2 \gamma_s = \frac{1+2s}{1+2\de}.$ 
Hence the flow $g_s(t)$ is defined, at least, up to time $\de,$ and we restrict to only considering $g_s(t)$ for times $t \in [0,\de].$
A computation yields that for $t \in [0,\de]$
\beq
	\label{forms_match}
		\frac{g_s(t)}{1+2t} = \frac{g( \gamma_s + (1+2\gamma_s)t)}{(1+2t)(1+2\gamma_s)} = \frac{g( \gamma_s + (1+2\gamma_s)t)}{1+2( \gamma_s +(1+2\gamma_s)t)}
\eeq
where $\gamma_s + (1 + 2 \gamma_s)t \leq \gamma_s + (1+2\gamma_s)\de = s \leq \Tau_{\max}.$
Therefore \eqref{good_barriers} tells us that $(1-b)\cH \leq \frac{g_s(t)}{1+2t} \leq (1+b)\cH$
throughout $\B_\cH \left( x , R-  \left\lfloor \frac{R}{\Lambda} \right\rfloor \Lambda +2 \right) \times [0, \de].$
Further, $\B _\cH \left( x , R-  \left\lfloor \frac{R}{\Lambda} \right\rfloor \Lambda +2 \right) \subset \B_\cH (x , R) \subset \subset \m$ by assumption.
Clearly $R-  \left\lfloor \frac{R}{\Lambda} \right\rfloor \Lambda +2 \geq 2$ and hence, recalling how $b$ was specified at the start of the proof, 
we may apply Lemma \ref{curvature_control} to the flow $g_s(t)$ to obtain that $-1 - \al \leq \K_{\frac{g_s(\de)}{1+2\de}} \leq -1 + \al$
throughout $\ovB_\cH \left( x , R-  \left\lfloor \frac{R}{\Lambda} \right\rfloor \Lambda \right).$
Using \eqref{forms_match} for $t=\de$ yields that $ \frac{g_s(\de)}{1+2\de} = \frac{ g(s)}{1+2s},$ 
and so the Gauss curvature control for $\frac{g_s(\de)}{1+2\de}$ tells us that $-1 - \al \leq \K_{\frac{g(s)}{1+2s}} \leq -1 + \al$
throughout $\ovB_\cH \left( x , R-  \left\lfloor \frac{R}{\Lambda} \right\rfloor \Lambda \right).$
Repeating for all $s \in [\de , \Tau_{\max}]$ allows us to conclude that $-1 - \al \leq \K_{\frac{g(s)}{1+2s}} \leq -1 + \al$
throughout $\ovB_\cH \left( x , R-  \left\lfloor \frac{R}{\Lambda} \right\rfloor \Lambda \right) \times [\de, \Tau_{\max}],$
as required in\eqref{Gauss_curv_bound_conc}.

If we are only assuming both the estimates in \eqref{over_arch_assumed_ests} for $g(0)$ throughout $\B_\cH (x,R)$ we stop here and are done.
If instead we are assuming $g(0) \equiv \cH$ throughout $\m,$ we make a final additional step to avoid any time delay before obtaining 
the Gauss curvature control claimed in \eqref{Gauss_curv_bound_conc}.
Indeed, we have that $(1-b)\cH \leq \frac{g(t)}{1+2t} \leq (1+b)\cH$ throughout $\B_\cH \left( x , R - \left\lfloor \frac{R}{\Lambda} \right\rfloor \Lambda + 2 \right) \times [0, \ep],$ 
and additionally we have $g(0) \equiv \cH$ throughout $\m$ by assumption. 
Recalling how $b$ was specified at the start of the proof, and noting that $R - \left\lfloor \frac{R}{\Lambda} \right\rfloor \Lambda + 2 \geq 2,$ 
we may appeal to Lemma \ref{no_time_delay} to conclude that $-1 - \al \leq \K_{\frac{g(t)}{1+2t}} \leq -1 + \al$ throughout
$\ovB_\cH \left(x , R - \left\lfloor \frac{R}{\Lambda} \right\rfloor \Lambda \right) \times [0, \ep].$
Combined with our previous Gauss curvature estimates, 
we obtain the Gauss curvature estimates in \eqref{Gauss_curv_bound_conc}
for all times $t \in [0, \Tau_{\max}],$ i.e. we have removed the time delay as required.
This completes the proof of Theorem \ref{over_arch_thm}.
\end{proof}

\begin{proof}[Proof of Theorem \ref{mainresult}]
Retrieve the universal constant $\ep > 0$ arising in Theorem \ref{over_arch_thm}. Let $\al \in (0,1]$ and $\de \in (0,\ep).$
Take $\Lambda = \Lambda ( \al , \de ) > 0$ and $b = b ( \al , \de ) > 0$ to be the respective constants arising in Theorem \ref{over_arch_thm}.
We may now define 
\beq
	\label{def_c}
		c = c(\al , \de ) := \frac{1}{4 \Lambda} \log (1+2\ep) > 0
\eeq
and 
\beq
	\label{radius_cR}
		\cR = \cR ( \al , \de ) := \max \left\{ \left( 1 + \frac{2}{\log(1+2\ep)} \right) \Lambda , 4\Lambda \frac{ \log ( 2 \sqrt{1+2\ep})}{\log (1+2\ep)} \right\} \geq \Lambda > 0.
\eeq
Now assume that $R \geq \cR$ and $(\m , \cH)$ is a smooth surface 
which satisfies that, for some $x \in \m,$ the ball $\B_\cH (x,R) \subset \subset \m$ and $\left( \B_\cH (x , R) , \cH \right)$ is isometric to a hyperbolic disc of radius $R.$ 
Suppose $g(t)$ is a complete smooth Ricci flow on $\m,$ defined for all $t \in [0,T]$ for some $T > 0,$ with $g(0)$ conformal to $\cH,$ and satisfying that
$(1-b) \cH \leq g(0) \leq (1+b) \cH$ and $| \K_{g(0)} | \leq 2$ throughout $\B_\cH (x,R).$ 
From \eqref{radius_cR} we have that $R \geq \cR \geq \Lambda.$ 
Therefore we may appeal to Theorem \ref{over_arch_thm} to obtain,
recalling \eqref{Gauss_curv_bound_conc} and \eqref{time_conc}, that at the point $x \in \m$ we have
$-1 - \al \leq \K_{\frac{g(t)}{1+2t}}(x) \leq -1 + \al$ for all times $\de \leq t \leq \tilde{\Tau}_{\max}$ where 
\beq
	\label{tilde_tau}
		\tilde{\Tau}_{\max} := \min \left\{ T , \frac{ 1}{2} \left( \exp \left[ \left\lfloor \frac{R}{\Lambda} \right\rfloor \log (1+2\ep) \right] - 1 \right) \right\}.
\eeq
Observe that \eqref{radius_cR} gives that $R \geq \cR \geq \left( 1 + \frac{2}{\log(1+2\ep)} \right) \Lambda.$
Therefore $\left( \frac{R}{\Lambda} - 1 \right) \log (1+2\ep) \geq 2$ and thus 
\beq
	\label{use1}
		\exp \left[ \left\lfloor \frac{R}{\Lambda} \right\rfloor \log (1+2\ep) \right] - 1 \geq \exp \left[ \left( \frac{R}{\Lambda} - 1 \right) \log (1+2\ep) \right] - 1 \geq \exp \left[ \frac{1}{2} \left( \frac{R}{\Lambda} - 1 \right) \log (1+2\ep) \right]
\eeq
since $e^x -1 \geq e^{\frac{x}{2}}$ for $x \geq 2.$

For $x , y > 0$ we have $\frac{1}{x}e^y \geq e^{\frac{y}{2}}$ provided $ y \geq 2 \log (x).$
Observe that $ R \geq \cR \geq   4\Lambda \frac{ \log ( 2 \sqrt{1+2\ep})}{\log (1+2\ep)}$ from \eqref{radius_cR}, 
and so $\frac{R}{2\Lambda} \log (1+2\ep) \geq 2 \log ( 2 \sqrt{1+2\ep}).$ 
Thus, using the above inequality with $x := 2\sqrt{1+2\ep}$ and $y := \frac{R}{2\Lambda} \log (1+2\ep),$ we deduce that
\beq
	\label{use2}
		\frac{1}{2\sqrt{1+2\ep}}\exp \left[ \frac{R}{2\Lambda} \log (1+2\ep) \right] \geq \exp \left[ \frac{R}{4\Lambda} \log(1+2\ep) \right] = e^{cR},
\eeq
recalling the definition of $c > 0$ in \eqref{def_c}.
Finally we can compute that
\begin{align*}
\tilde{\Tau}_{\max} &\stackrel{\eqref{tilde_tau}}{=} \min \left\{ T , \frac{ \exp \left[ \left\lfloor \frac{R}{\Lambda} \right\rfloor \log (1+2\ep) \right] - 1}{2} \right\} 
			\stackrel{\eqref{use1}}{\geq} \min \left\{ T , \frac{1}{2}\exp \left[ \frac{1}{2} \left( \frac{R}{\Lambda} -1 \right) \log (1+2\ep) \right] \right\} = \ldots \\
		&\quad \ldots = \min \left\{ T , \frac{1}{2\sqrt{1+2\ep}}\exp \left[  \frac{R}{2\Lambda} \log (1+2\ep) \right] \right\} \stackrel{\eqref{use2}}{\geq} \min \left\{ T , e^{cR} \right\} =: \Tau_{\max}
\end{align*}
as claimed in \eqref{conc_of_main_thm_2} in Theorem \ref{mainresult}.
\end{proof}

\begin{proof}[Proof of Theorem \ref{hypat0}]
Retrieve the universal constant $\ep > 0$ arising in Theorem \ref{over_arch_thm}. Let $\al \in (0,1]$ be given and take $\de := \frac{\ep}{2} \in (0, \ep).$
For this choice of $\de$ we can retrieve constants $\Lambda = \Lambda (\al ) > 0$ and $ b = b ( \al ) > 0$ from Theorem \ref{over_arch_thm}.
Using these constants, we can define $c > 0$ and $\cR > 0$ exactly as they are defined in \eqref{def_c} and \eqref{radius_cR} respectively,
now both depending only on $\al$ as required.
Repeat the proof of Theorem \ref{mainresult}, observing that, in the notation of Theorem \ref{over_arch_thm}, we now assume that $g(0) \equiv \cH$ throughout $\m,$
and so we may now use the version of Theorem \ref{over_arch_thm} that avoids any time delay before achieving the desired Gauss curvature control.
Proceeding verbatim as in the proof of Theorem \ref{mainresult} above establishes that we have the Gauss curvature estimates claimed in \eqref{conc_time_0_equiv_thm} at $x \in \m$
for the time required in \eqref{conc_time_0_equiv_thm} in Theorem \ref{hypat0}.
\end{proof}

\noindent
{\sc Mathematical Institute,
University of Oxford,
Andrew Wiles Building
Radcliffe Observatory Quarter (550),
Woodstock Road,
Oxford,
OX2 6GG }

\noindent
{\bf{AM:}} \emph{andrew.mcleod@maths.ox.ac.uk} 

\noindent
\url{https://www.maths.ox.ac.uk/user}

\end{document}